\documentclass[10pt]{amsart}
\usepackage[a4paper, total={6in, 9in}]{geometry}

\usepackage[dvipsnames]{xcolor} 
\usepackage[urlcolor=pink, 
linkcolor=blue,
citecolor=Emerald,
colorlinks=true,
linktocpage=true]{hyperref}
\usepackage{tikz}
\usepackage{graphicx,amssymb,eucal,mathrsfs}
\usepackage{latexsym,amsmath,epsfig,epic,eepic}
\usepackage{amscd}
\usepackage{amsthm}
\usepackage{indentfirst}
\usepackage{epsf,graphicx,pgf}
\usepackage{pstricks,mathrsfs}

\title[The sharp~$C^0$-fragmentation property]{The sharp~$C^0$-fragmentation property for Hamiltonian diffeomorphisms and homeomorphisms on surfaces}
\author{Baptiste Serraille}

\newtheorem{thm}{Theorem}

\newtheorem{prop}{Proposition}[section]
\newtheorem*{prop*}{Proposition}
\newtheorem{lemma}[prop]{Lemma}

\newtheorem{rk}[prop]{Remark}
\newtheorem{dfn}[prop]{Definition}

\newtheorem{coro}{Corollary}
\newtheorem{ques}{Question}[section]

\newcommand{\Sympoc}{\mathrm{Symp}_{0}}
\newcommand{\Sympo}{\mathrm{Symp}_{0}}
\newcommand{\Sympc}{\mathrm{Symp}}
\newcommand{\Symp}{\mathrm{Symp}}
\newcommand{\Ham}{\mathrm{Ham}}
\newcommand{\Homeo}{\mathrm{Homeo}}
\newcommand{\Ker}{\mathrm{Ker}}
\newcommand{\Flux}{\mathrm{Flux}}
\newcommand{\wFlux}{\widetilde{\mathrm{Flux}}}

\newcommand\restr[2]{{
  \left.\kern-\nulldelimiterspace 
  #1 
  \vphantom{|} 
  \right|_{#2} 
  }}

\begin{document}

\begin{abstract}
In this paper, we present a~$C^0$-fragmentation property for Hamiltonian diffeomorphisms. More precisely, it is known that for a given open covering~$\mathcal{U}$ of a compact symplectic surface we can write each~$C^0$-small enough Hamiltonian diffeomorphism as the composition of Hamiltonian diffeomorphisms compactly supported inside the open sets of the covering~$\mathcal{U}$. We show that such a decomposition can be done with a Lipschitz estimate on the~$C^0$-norm of the fragments. We also show the same property for the kernel of~$\theta$, the mass-flow homomorphism for homeomorphisms. This answers a question from Buhovsky and Seyfaddini.
\end{abstract}

\maketitle

\tableofcontents

\section{Introduction and main results}\label{sec: intro}

\subsection{The~$C^0$-fragmentation property on surfaces}\label{subsec: C^0 frag for diffeo}

\subsubsection{The $C^0$-fragmentation for the group of Hamiltonian diffeomorphisms} \label{subsubsec: C^0-frag for diffeo}

The fragmentation property on a given manifold~$M$ allows to decompose diffeomorphisms (or homeomorphisms) of various kinds into a composition of diffeomorphisms (homeomorphisms) supported in small balls. The more refined~$C^0$-fragmentation property provides a control on the size of the fragments. We will be interested in fragmenting elements of the group of Hamiltonian diffeomorphisms on a surface~$\Sigma$. We denote $\Ham(X, \omega)$ the group of compactly supported Hamiltonian diffeomorphisms on $(X, \omega)$. We will sometimes drop the symplectic form $\omega$ in the notation if it is clear which symplectic form is used.

One motivation for having~$C^0$-fragmentation is the following corollary that gives us the existence of paths in the group of Hamiltonian diffeomorphisms of a symplectic surface such that their $C^0$-norm increase in a way that is precisely controlled.

\begin{coro}\label{coro: $C^0$-small isotopy}
	Let~$(\Sigma, \omega)$ be a closed symplectic surface and~$d$ the distance induced by some Riemannian metric on $\Sigma$. There exists a constant~$C>0$, such that for all~$\phi$ in the group of Hamiltonian diffeomorphisms, there exists a Hamiltonian isotopy~$\{\phi_t\}$ such that~$\phi_0=Id$,~$\phi_1=\phi$, that satisfies the following estimate for all~$t \in [0,1]$:
	$$\Vert \phi_t \Vert_{C^0}\leq C \Vert \phi \Vert_{C^0}.$$
\end{coro}

The fragmentation property has been introduced  in \cite{Thurston}, \cite{Ban78} and \cite{Ban97} by Thurston and Banyaga in order to study the simplicity and perfectness of the groups of diffeomorphisms preserving a symplectic form or preserving a volume form. Later, in his study of the structure of the group of homeomorphisms preserving a full nonatomic measure, Fathi proved in \cite{Fathi} a similar fragmentation property for measure-preserving homeomorphisms. This led to the proof of the simplicity of the group of compactly supported measure-preserving homeomorphisms of the ball of dimension~$n$ whenever~$n\geq 3$. The case of the dimension 2 turns out to be much harder. Only recently Cristofaro-Gardiner, Humilière and Seyfaddini proved that there exists a normal subgroup of the group of area-preserving homeomorphisms of the disk in \cite{CriHumSey}. In earlier work (\cite{LeRoux}), Le Roux showed that the simplicity of the group of area-preserving homeomorphism of the disk is equivalent to the existence of some fragmentation property for homeomorphisms. After Le Roux's work Entov, Polterovich and Py and then Seyfaddini proved quantitative versions of the fragmentation property of~$C^0$-small Hamiltonian diffeomorphisms preserving a volume form in the 2 dimensional case, see \cite[Section 1.6.2]{EPP} and \cite[Proposition 3.1]{Seyfaddini}. On the one hand, the goal in \cite{EPP} was to construct quasi-morphisms on the group of diffeomorphisms, thus understanding better the algebraic properties of this group and on the other hand, Seyfaddini gave a Hölder-type bound on the~$C^0$-norm of the fragments in order to show the~$C^0$-continuity of the Oh-Schwarz spectral invariants. In the present paper we adapt their proof to show the sharpest~$C^0$-fragmentation possible, that is a Lipschitz bound on the~$C^0$-norm of the fragments. This is the content of our first theorem.

\begin{thm}\label{thm: fragmentation lemma on surfaces with Lipschitz estimate}
	Let~$(\Sigma,\omega)$ be a closed surface equipped with an area form~$\omega$ and let~$d$ be a distance induced by some Riemannian metric. Let~$\mathcal{W}=(W_i)_{i=1}^m$ be a finite open covering of the surface~$\Sigma$ by disks. Then there exists a~$C^0$-neighborhood~$N$ of the identity in the group~$\Ham(\Sigma)$ of Hamiltonian diffeomorphisms such that for each~$\phi\in N$, we can decompose~$\phi$ as:
	$$\phi=\phi_1 \circ \phi_2 \circ \cdots \circ \phi_m,$$
	where for all~$1\leq i\leq m$,~$\phi_i$ belongs to~$\Ham(W_i)$. Moreover, we have the following estimate for all~$1\leq i\leq m,$
	$$\Vert \phi_i \Vert_{C^0}\leq C \Vert \phi \Vert_{C^0},$$
	for some constant~$C>0$ independent of~$\phi$.
\end{thm}

\begin{rk} 
	The result proved by Seyfaddini in \cite[Proposition 3.1]{Seyfaddini} is, for a very specific covering given by the handle decomposition of the surface, the estimate:~$$\Vert \phi_i \Vert_{C^0}\leq C \left \Vert\phi\right \Vert_{C^0}^{2^{1-N}},$$
	where~$N= 2g+2$ depends on the genus~$g$ of the surface.
\end{rk}

\begin{rk}
It would be interesting whether such $C^0$-fragmentation property hold in higher dimension, the methods used in this paper are specific to area-preserving diffeomorphisms and cannot be adapted directly in higher dimension. However, we point out that fragmentation properties in dimension 4 do exist, as proven by Alizadeh in \cite{Ali}.
\end{rk}

The fragmentation property stated in Theorem \ref{thm: fragmentation lemma on surfaces with Lipschitz estimate} relies on Lemma \ref{lem: new version of the extension lemma}, which is an improvement of the extension lemma in \cite[Section 3.4.2]{Seyfaddini}.

In the following lemma and for the rest of the paper we will denote by~$\mathbb{A}_y$ the annulus~$S^1\times [-y,y]$ for~$y$ a positive real number and we will also denote~$\mathbb{A}_{y,y'}:=S^1\times ([-y',-y] \cup [y,y'])$ if~$y' >y$. We recall that for a subset~$C$ of a manifold~$M$, the notation~$Op(C)$ means any open neighborhood of~$C$ in~$M$.

\begin{lemma}[Area-preserving extension lemma for the annulus]\label{lem: new version of the extension lemma}
	Consider the annulus~$\mathbb{A}_2$ with a symplectic form~$\omega$. Let~$\mathcal{E}$ be the set of smooth area-preserving embeddings~$\phi: Op(\mathbb{A}_1)\rightarrow \mathbb{A}_2$ which are homotopic to the inclusion and such that for some~$y\in (-1,1)$, and hence for all, the signed area in~$\mathbb{A}_2$ bounded by~$S^1\times y$ and~$\phi(S^1 \times y)$ is zero.
	
	Then there exists~$\delta, D, C>0$, such that for all~$\phi \in \mathcal{E}$ with~$\Vert\phi \Vert_{C^0}\leq \delta$, there exists~$\psi \in \text{Ham}(\mathbb{A}_2)$ such that~$\psi \vert_{\mathbb{A}_{1-D\Vert\phi \Vert_{C^0}}}=\phi\vert_{\mathbb{A}_{1-D\Vert\phi \Vert_{C^0}}}$ and 
	$$\Vert \psi \Vert_{C^0}\leq C \Vert\phi \Vert_{C^0}.$$
	
	Moreover, if for some arc~$I\subset S^1$ we have that~$\phi=Id$ outside a quadrilateral~$I\times[-1,1]$ and~$\phi(I\times [-1,1])\subset I\times [-2,2]$, then the extension~$\psi$ can be chosen to be the identity outside of~$I\times [-2,2]$.
\end{lemma}

\begin{rk}
In \cite{Seyfaddini}, Seyfaddini extended the method developed by Entov, Polterovich and Py to show the bound
$$\Vert \psi \Vert_{C^0}\leq C \left( \Vert \phi \Vert_{C^0}\right) ^{\frac{1}{2}}$$
for some constant~$C>0$. The method employed then cannot give a better estimate than the one he obtained.
\end{rk}


\subsubsection{The $C^0$-fragmentation for the kernel of the mass-flow homomorphism}\label{subsubsec: C^0 frag for homeo}

As in \cite{Fathi}, it is sensible to ask whether the $C^0$-fragmentation holds also for homeomorphisms preserving a "good" measure. It is necessary then to restrict our study to a continuous analogue of Hamiltonian diffeomorphisms, that is the group $\Ham(\Sigma)$ is replaced by the group $\Ker(\theta)$. The map $\theta : \Homeo_0(\Sigma, \omega) \to
H^1(\Sigma,\mathbb R)/\Gamma$, where $\Gamma$ is some discrete subgroup of $H^1(\Sigma;\mathbb R)$, denotes the mass-flow homomorphism, first defined in \cite{Schwartzman}. The map originally had target $H_1(\Sigma,\mathbb R)/\Gamma$. We defined the map on the dual to match better the definition of the $\Flux$. Under the assumption that homeomorphism lies in $\Ker(\theta)$ the fragmentation property still holds.

\begin{thm}\label{thm: fragmentation lemma for Homeo}
	Let~$(\Sigma,\omega)$ be a closed surface equipped with an area form $\omega$ and~$d$ be a distance induced by a Riemannian metric. Let~$\mathcal{W}=(W_i)_{i=1}^m$ be a finite open covering of the surface by disks. Then there exists a~$C^0$-neighborhood~$N$ of the identity in the group~$\Ker(\theta)$ such that for each~$\phi\in N$, we can decompose~$\phi$ as:
	$$\phi=\phi_1 \circ \phi_2 \circ\cdots \circ \phi_m,$$
	where for all~$1\leq i\leq m$,~$\phi_i$ belongs to~$\Homeo_{c}(W_i,\omega)$, the group of compactly supported homeomorphisms preserving the area in~$W_i$. Moreover, we have the following estimate for all~$1\leq i\leq m,$
	$$\Vert \phi_i \Vert_{C^0}\leq C \Vert \phi \Vert_{C^0},$$
	for some constant~$C>0$ independent of~$\phi$.
\end{thm}

\begin{rk}
We point out that the theorem above is not proven by a limit process (by this we mean, writing $\phi$ as the $C^0$-limit of Hamiltonian diffeomorphisms and choosing their respective fragmentations to converge to a fragmentation of $\phi$ by Hamiltonian homeomorphisms) but rather by adapting all the proofs of the smooth case.
\end{rk}

In the same spirit as for Corollary \ref{coro: $C^0$-small isotopy}, we show that for all homeomorphisms~$\phi$ in~$\Ker(\theta)$ there exists a~$C^0$-small isotopy in~$\Ker(\theta)$ between~$Id$ and~$\phi$.

\begin{coro}\label{coro: $C^0$-small isotopy for homeo}
	Let~$(\Sigma, \omega)$ be a closed symplectic surface and~$d$ be a distance induced by some Riemannian metric. There exists a constant~$C>0$, such that for all~$\phi$ in~$\Ker(\theta)$, there exists an isotopy~$\{\phi_t\}$ of area-preserving homeomorphisms in~$\Ker(\theta)$ such that~$\phi_0=Id$,~$\phi_1=\phi$, that satisfies the following estimate for all~$t \in [0,1]$:
	$$\underset{t}{\sup}\Vert \phi_t \Vert_{C^0}\leq C \Vert \phi \Vert_{C^0}.$$
\end{coro}

We will prove Theorem \ref{thm: fragmentation lemma for Homeo} by using an analog of Lemma \ref{lem: new version of the extension lemma} for homeomorphisms.

\begin{lemma}\label{lem: area-preserving extension lemma for homeo}
	Consider~$\mathbb{A}_2$ with an area form $\omega$. Let~$\mathcal{E}'$ be the set of continuous measure-preserving embeddings~$\phi: Op(\mathbb{A}_1)\rightarrow \mathbb{A}_2$ which are homotopic to the inclusion and such that for some~$y\in (-1,1)$, and hence for all, the signed area in~$\mathbb{A}_2$ bounded by~$S^1\times y$ and~$\phi(S^1 \times y)$ is zero.
	
	Then there exist~$\delta, D , C>0$, such that for all~$\phi \in \mathcal{E}'$ with~$\Vert\phi \Vert_{C^0}\leq \delta$, there exists a compactly supported area-preserving homeomorphism~$\psi \in \ker(\theta)$ of $\mathbb A_2$ such that~$\psi \vert_{\mathbb{A}_{1-D\Vert\phi \Vert_{C^0}}}=\phi\vert_{\mathbb{A}_{1-D\Vert\phi \Vert_{C^0}}}$ and 
	$$\Vert \psi \Vert_{C^0}\leq C \Vert\phi \Vert_{C^0}.$$
	Moreover, if for some arc~$I\subset S^1$ we have that~$\phi=Id$ outside a quadrilateral~$I\times[-1,1]$ and~$\phi(I\times [-1,1])\subset I\times [-2,2]$, then the extension~$\psi$ can be chosen to be the identity outside of~$I\times [-2,2]$.
\end{lemma}


\subsection{Organization of the paper}\label{subsec: Organization}

We review briefly the content of each section. In Section \ref{sec: proof of the corollaries} we prove Corollary \ref{coro: $C^0$-small isotopy} and \ref{coro: $C^0$-small isotopy for homeo} assuming that we know the $C^0$-fragmentation property (Theorem \ref{thm: fragmentation lemma on surfaces with Lipschitz estimate} and Theorem \ref{thm: fragmentation lemma for Homeo}). In Section \ref{sec: pf of the frag prop} we show that the area-preserving extension lemmas yield the $C^0$-fragmentation property for diffeomorphisms and for homeomorphisms. In Section \ref{sec: pf of the area-preserving extension lemma} we prove the area-preserving extension lemmas and in Section \ref{sec: pf of extension lemma} we prove an extension lemma that we use in Section \ref{sec: pf of the area-preserving extension lemma}.


\subsection{Acknowledgments}\label{subsec: acknowledgments}

I would like to warmly thank Lev Buhovsky for asking me this question and hosting me in Tel-Aviv University, he also pointed out to me the idea behind Section \ref{sec: pf of the frag prop}. I would also like to thank him for introducing me to this beautiful field that symplectic geometry is. I thank
the anonymous referee for helpful comments as well as for spotting a mistake in the first version of this paper. I also thank Lev Buhovsky, Emmanuel Giroux, Sobhan Seyfaddini and Maksim Stoki\'{c} for giving me great advice related to the redaction of this paper and even reading some parts of the paper with me. I want to thank the members of the student lab of Tel-Aviv University and Menashe for their hospitality that made my stay very pleasant. I thank the numerous people of the workshop
of symplectic topology in Belgrade and Tristan Humbert that sparked my interest of the question of the optimal constant in Corollary \ref{coro: $C^0$-small isotopy}. I also thank Dylan Cant for a fruitful discussion that led to the annex of this paper. I was partially supported by ERC Starting Grant 757585.


\section{Some definitions and notations}

In this section we go through some definitions and notations that will be used later on and describe geometrical interpretations of the flux homomorphism and the mass-flow homomorphism.

Let~$(M, \omega)$ be a connected symplectic manifold. Let~$\Symp(M,\omega)$ be the set of compactly smooth diffeomorphisms of~$M$ that preserve~$\omega$. We then let~$\Sympo(M, \omega)$ be the connected component of~$Id$ in~$\Symp(M,\omega)$, i.e.~$\phi \in \Sympo(M,\omega)$ if and only if there exists a smooth family of compactly supported symplectic diffeomorphisms~$\left(\phi_t\right)_{t\in[0,1]}$ such that~$\phi_0=Id$ and~$\phi_1=\phi$.

A \textit{Hamiltonian}~$H$ is a smooth function~$H: [0,1] \times M \rightarrow \mathbb{R}$ compactly supported in~$[0,1]\times M$. A Hamiltonian~$H$ induces the \textit{Hamiltonian flow}
$$\phi_H^t: M\rightarrow M,\quad (0\leq t \leq 1),$$ 
by integrating the unique time-dependent vector field~$X_H$ satisfying~$ \iota_{X_H}\omega=dH_t$, the isotopy $\phi_H^t$ is a symplectic isotopy. A \textit{Hamiltonian diffeomorphism} is a diffeomorphism obtained as the time-1 map of a Hamiltonian flow. We will denote~$\text{Ham}(M,\omega) \subset \Symp_0(M,\omega)$ the set of such diffeomorphisms. We will eliminate the symplectic form~$\omega$ from the above notation when no confusion is possible. We recall that if~$H$ and~$G$ are two Hamiltonians with flows~$\phi_H^t$ and~$\phi_G^t$, then the Hamiltonian~$H \# G(t,x)=H(t,x)+G(t,(\phi_H^t)^{-1}(x))$ generates the Hamiltonian flow~$\phi_H^t \circ \phi_G^t$.

We will now give an alternative definition of~$\Ham(M)$ using the~$\Flux$, we will need this definition later on, for a proof of the properties stated below we refer the reader to \cite[Chapter 10]{McD-Sal}. We define first the homomorphism
$$\wFlux:  \widetilde{\Sympo}(M,\omega)\rightarrow H^1(M,\mathbb{R}),$$
where \smash[t]{$\widetilde{\Sympo}(M,\omega)$} is the universal cover of~$\Sympo(M,\omega)$ in the following way. Let~$\{\phi_t\}$ be a symplectic isotopy from~$Id$ to~$\phi_1$ and let~$X_t$ be the time-dependent vector field defined via the relation~$$\dfrac{d}{dt}\phi_t=X_t\circ \phi_t.$$
Since~$\{\phi_t\}$ is a symplectic isotopy, the 1-form~$\iota_{X_t}\omega$ is a closed form. We then define:
$$\wFlux(\phi_t):=\int_0^1 [\iota_{X_t}\omega]\,dt \in H^1(M,\mathbb{R}).$$

This 1-form depends only on the choice of the homotopy class of the isotopy~$\phi_t$ with fixed endpoints so \smash[t]{$\wFlux$} is well-defined. Also one can see by the natural identification of~$H^1(M,\mathbb{R})$ and~$\text{Hom}(\pi_1(M),\mathbb{R})$ that \smash[t]{$\wFlux(\{\phi_t\})$} acts on~$\pi_1(M)$. We can see that roughly speaking, in dimension 2 this action describes how much ``mass" goes through a loop~$\gamma$ in~$M$ during the isotopy. If~$\{\phi_t\}$ is a Hamiltonian isotopy, then the 1-form~$\iota_{X_t}\omega$ is exact and thus~$\smash[t]{\wFlux(\phi_t)=0}$, moreover one can show that if~$\phi_t$ is such that\smash[t]{~$\wFlux(\phi_t)=0$} then~$\phi$ is a Hamiltonian diffeomorphism. A proof of this non-trivial fact can again be found in \cite[Chapter 10]{McD-Sal}.

We let $\Gamma$ the image under $\smash[t]{\wFlux}$ of the loops based at $Id$ in $\Sympo(M, \omega)$, then the homomorphism $\smash[t]{\wFlux}$ descends to $\Flux : \Sympo(M) \to H^1(M, \mathbb R)/\Gamma$. This way we have a reciprocal statement, a symplectomorphism $\phi$ in $\Sympo(M)$ is in $\Ham(M)$ if and only if we have $\Flux(\phi) = 0$.

\bigskip

On a symplectic manifold~$(M,\omega)$ endowed with a distance~$d$ induced by a Riemannian metric, the \textit{$C^0$-norm} (or \textit{uniform norm}) is defined by
$$\Vert \phi \Vert_{C^0}:= \underset{x}{\max} ~d\left(x,\phi(x)\right).$$
Similarly, given a symplectic isotopy~$\{\phi_t\}$, we define its \textit{$C^0$-norm} by
$$\Vert\{\phi_t\}\Vert_{C^0}^{\text{path}}:=\underset{x,t}{\max} ~d(x,\phi_t(x))=\underset{t}{\sup}\Vert \phi_t \Vert_{C^0}.$$
This norm induces what is called the \textit{$C^0$-topology}. We denote
$$\overline{\Ham}(M, \omega) \subset \Homeo_0(M,\omega)$$ the closure for the~$C^0$-norm of the group of Hamiltonian diffeomorphisms in the group of homeomorphisms
of $M$ in the identity component of the identity. Note that in the particular case where $M$ is a surface every diffeomorphism of $\overline{\Ham}(M, \omega$) is in $\Ham(M, \omega)$.

\bigskip

In this paragraph we focus on the case of a closed surface~$(\Sigma,\mu)$ equipped with a measure~$\mu$. We define~$\Homeo(\Sigma,\mu)$ as the set of homeomorphisms of~$\Sigma$ that preserve the measure~$\mu$ and~$\Homeo_0(\Sigma,\mu)$ the identity component of~$\Homeo(\Sigma,\mu)$. The \textit{mass-flow homomorphism} was introduced by Schwartzman in \cite{Schwartzman}. For the definition of the mass-flow homomorphism on general compact metric spaces we refer to \cite[Section 0]{Fathi}. Let \smash[t]{$\widetilde{\Homeo}(\Sigma, \mu)$} be the set of paths starting at the identity in~$\Homeo_{0}(\Sigma,\mu)$. Fathi first defines a homeomorphism~\smash[t]{$\widetilde{\theta}: \widetilde{\Homeo_{0}}(\Sigma, \mu)\rightarrow H^1(\Sigma,\mathbb{R})$}. Denote by~$\Gamma$ the image under \smash[t]{$\widetilde{\theta}$} of the subset of loops based at~$Id$ in~$\Homeo(\Sigma)$, then~$\widetilde{\theta}$ descends naturally to~$\theta: \Homeo(\Sigma,\mu)\rightarrow H^1(\Sigma,\mathbb{R})/\Gamma$.  It is proved in \cite{Lefeuvre} that if~$\mu$ is induced by a symplectic form, then~$\overline{\Ham}(\Sigma,\mu)=\Ker(\theta)$. Theorem \ref{thm: fragmentation lemma for Homeo} can thus be rewritten in terms of~$\overline{\Ham}(\Sigma,\mu)$ whenever~$\mu$ is given by a symplectic form.


\section{Proof of some consequences of the $C^0$-fragmentation property}\label{sec: proof of the corollaries}

In this section we prove the Corollaries \ref{coro: $C^0$-small isotopy} and \ref{coro: $C^0$-small isotopy for homeo} of Theorems \ref{thm: fragmentation lemma on surfaces with Lipschitz estimate} and \ref{thm: fragmentation lemma for Homeo}. In order to prove Corollary \ref{coro: $C^0$-small isotopy} we will use the case where~$\Sigma$ is a disk, the following result is proved by Seyfaddini in \cite[Lemma 3.2]{Seyfaddini}, the proof relies on an ingenious use of Alexander's trick.

\begin{lemma}\label{lem: clever Alexander's trick}
	We endow $\mathbb R ^{2n}$ with the standard metric. Suppose~$\psi \in \Ham(B_r^{2n})$,  then there exists a Hamiltonian~$H: [0,1] \times B_r^{2n} \rightarrow \mathbb{R}$ such that~$\psi=\phi_H^1$ and~$\underset{t}{sup} \Vert \phi_H^t \Vert_{C^0}\leq \Vert \psi \Vert_{C^0}$.
\end{lemma}

With the help of Lemma \ref{lem: clever Alexander's trick}, one can prove Corollary \ref{coro: $C^0$-small isotopy}.

\begin{proof}[Proof of Corollary \ref{coro: $C^0$-small isotopy}]

Assume that we know that the statement of Corollary \ref{coro: $C^0$-small isotopy} is true on a~$C^0$-neighborhood of~$Id$ in~$\Ham(\Sigma)$. That is, there exists~$\varepsilon>0$ and a constant~$C>0$, such that for all~$\Vert\phi\Vert_{C^0}\leq \varepsilon$, there exists a Hamiltonian~$H_t$ such that~$\Vert\phi_H^t\Vert_{C^0}^{\text{path}} \leq C \Vert \phi \Vert_{C^0}$. We note~$D:=\text{Diam}(\Sigma)$ the diameter of the surface, it is well-defined by compactness. Then~$\Vert \phi_H^t \Vert_{C^0}^{\text{path}}\leq D$ for any Hamiltonian isotopy~$\phi_H^t$ with~$\phi_H^1=\phi$. We can compute that if~$\Vert \phi \Vert_{C^0} \geq \varepsilon$,
$$\Vert \phi_H^t \Vert_{C^0}^{\text{path}}\leq D =\frac{D}{\varepsilon} \varepsilon \leq \frac{D}{\varepsilon} \Vert \phi \Vert_{C^0}.$$
We can combine the two cases~$\Vert \phi \Vert_{C^0} \leq \varepsilon$ and~$\Vert \phi \Vert_{C^0} \geq \varepsilon$ to finally obtain that, for all~$\phi \in \Ham(\Sigma)$ there exists  a Hamiltonian~$H$ such that 
$$\Vert \phi_H^t \Vert_{C^0}^{\text{path}}\leq \max\left(\dfrac{D}{\varepsilon},C\right) \Vert \phi \Vert_{C^0}.$$
Which is exactly the result we wanted.

It remains to prove the property of Corollary \ref{coro: $C^0$-small isotopy} on a well-chosen~$C^0$-neighborhood of the identity. Take any finite covering~$\mathcal{U}=(U_i)_{i=1}^m$ of~$\Sigma$ by disks~$U_i$. We will take ~$N$ and~$C>0$ respectively the neighborhood and the constant given by Theorem \ref{thm: fragmentation lemma on surfaces with Lipschitz estimate}. Let~$\phi \in \Ham(\Sigma)$, the~$C^0$-fragmentation property tells us that it is possible to decompose~$\phi = \phi_1 \circ \phi_2 \circ \cdots \circ \phi_m$, where each~$\phi_i$ is a compactly supported Hamiltonian diffeomorphism of the disk~$U_i$ and~$\Vert \phi_i \Vert_{C^0}\leq C \Vert \phi \Vert_{C^0}$. We now use Lemma \ref{lem: clever Alexander's trick}, so there exist compactly supported Hamiltonians~$H_i: [0,1] \times U_i \rightarrow \mathbb{R}$ such that 
$$\phi_{H_i}^1=\phi_i, \quad \Vert \phi_{H_i}^t \Vert_{C^0}^{\text{path}}\leq C_i \Vert \phi_i \Vert_{C^0},$$
where $C_i$ is a constant that comes from the fact that the metric on $U_i$ is not the standard metric. The constant $C_i$ only depends on the metric on $U_i$ and not on $\phi$.

Let~$H:=H_1 \# H_2 \# \cdots \# H_m$, then~$H$ generates~$\phi_H^t$ a Hamiltonian isotopy such that~$\phi_{H}^t=\phi_{H_1}^t\circ \phi_{H_2}^t\circ \cdots \phi_{H_m}^t$, in particular~$\phi=\phi_H^1$. Also, by the property of the~$C^0$-norm and the previous estimates,
\begin{align*}
	\Vert\phi_{H}^t \Vert_{C^0}^{\text{path}}&=\Vert\phi_{H_1}^t\circ \phi_{H_2}^t\circ \cdots \circ \phi_{H_m}^t \Vert_{C^0}^{\text{path}}\\
	&\leq \Vert \phi_1^t \Vert_{C^0}^{\text{path}} +\Vert \phi_2^t \Vert_{C^0}^{\text{path}}+\cdots + \Vert \phi_m^t \Vert_{C^0}^{\text{path}}\\
	&\leq C_1 \Vert \phi_1\Vert_{C^0}+C_2 \Vert \phi_2 \Vert_{C^0}+\cdots + C_m\Vert \phi_m \Vert_{C^0}\\
	&\leq C(C_1+C_2+\ldots+C_m) \Vert \phi \Vert_{C^0}.
\end{align*}
Which is the result we intended to prove.
\end{proof}

We also prove Corollary \ref{coro: $C^0$-small isotopy for homeo} by using the same method, we just need to adapt Seyfaddini's method to show the following lemma.

\begin{lemma}\label{lem: clever Alexander's trick for homeo}
	Suppose that~$\psi \in \overline{\Ham}(B_r^{2n})$, then there exists a family~$\phi^t$ of homeomorphisms in~$\smash[t]{\overline{\Ham}(B_r^{2n})}$ such that~$\psi=\phi^0$,~$\psi=\phi^1$ and~$\underset{t}{sup} \Vert\phi^t \Vert_{C^0}\leq \Vert \psi \Vert_{C^0}$.
\end{lemma}

\begin{proof}
	This lemma is actually easier to prove than Lemma \ref{lem: clever Alexander's trick} since we do not have to take care of the smoothness of the family~$\phi^t$. We assume first that~$r=1$.
	We want to show that the family 
	$$\phi^s(x)=\begin{cases} s\psi(\frac{x}{s}) \text{ if } \vert x \vert \leq s,\\
		x \text{ otherwise}
	\end{cases}$$
	satisfies the conclusion of Lemma \ref{lem: clever Alexander's trick for homeo}. It is easily seen that~$\Vert\phi^t\Vert\leq t \Vert \psi \Vert_{C^0}$ we now only need to check that at all time~$\psi^t$ is indeed in the group~$\overline{\Ham}(B_r^{2n})$.
	
	Let~$(\psi_k)_{k\in \mathbb{N}}$ be a family of Hamiltonian diffeomorphisms generated by the Hamiltonians~$(H_k)_k$ approximating~$\psi$ for the~$C^0$-norm. Then,~$\phi^s$ can be approximated by~$\psi_k^s$, where~$\psi_k^s$ is defined as
	$$\psi_k^s(x)=\begin{cases}s \psi_k(\frac{x}{s}) \text{ if } \vert x \vert \leq s,\\
		x \text{ otherwise.}
	\end{cases}$$
	However, for all~$s>0$,~$\psi_k^s$ is generated by the Hamiltonian
	$$ H_{k,s}(t,x)=\begin{cases} s^2 H_k(t,\frac{x}{s}) \text{ if } \vert x \vert \leq s,\\
	0 \text{ otherwise.}
	\end{cases}$$
	Hence, proving that~$\phi^s$ is indeed approximated by Hamiltonian diffeomorphisms and thus finishing the proof of the lemma.
\end{proof}


\section{Proof of the $C^0$-fragmentation property}\label{sec: pf of the frag prop}

In this section we prove Theorem \ref{thm: fragmentation lemma on surfaces with Lipschitz estimate} and Theorem \ref{thm: fragmentation lemma for Homeo}. In order to prove the~$C^0$-fragmentation property it will be easier to work on a refinement of the given open covering~$\mathcal{W}$, we will first prove the~$C^0$-fragmentation property for this refinement and then go back to the initial covering by applying a trick that allows us to switch places of diffeomorphisms in the decomposition despite the lack of commutativity a priori. The refinement of the covering is obtained by associating an open set to each vertex, edge and face of a given triangulation~$T$, if the triangulation is taken thin enough the resulting covering will indeed be a refinement. This construction is described in Section \ref{subsec: def of covering} and is due to Thurston and Banyaga. In Section \ref{subsec: pf of the thm} we will decompose any~$C^0$-small diffeomorphism according to the new covering by defining the fragments we need in this order: around the vertices, the edges and then the faces. The construction of the fragments is where Lemma \ref{lem: new version of the extension lemma} will be useful. More precisely in Section \ref{subsec: 2 coro} we will describe two slight modifications of the extension lemma, namely the area-preserving extension lemma for disks and the area-preserving extension lemma for rectangles. It is however not possible to apply naively the two extension lemmas, there is some obstruction to it given by the area between a curve and the image of it (see. the area condition in Lemma \ref{lem: new version of the extension lemma}), this area represents an obstruction to the extension, we will take care of it in Section \ref{subsec: def of the obstr}.


\subsection{Covering associated to a triangulation}\label{subsec: def of covering}

In this section we associate to a triangulation~$T=(\Delta_i^k)_{i\in I_k}$,~$k=0,1,2$ ($\Delta_k^i$ is a face of dimension $k$ of the triangulation) three open coverings~$\mathcal{V}=(V_i^k)_{i\in I_k},\,\mathcal{V}'=(V_i^{'k})_{i\in I_k}$ and~$\mathcal{U}=(U_i^k)_{i\in I_k}$, following the construction of Thurston and Banyaga \cite{Ban78}. The covering~$\mathcal{V}$ verify the fact that for each~$k=0,1,2$ the elements of the subset~$(V_i)_{i\in I_k}^k$ are pairwise disjoint,~$\mathcal{U}$ and~$\mathcal{V'}$ are obtained by successively thickening~$\mathcal{V}$ and will also verify the same condition as~$\mathcal{V}$.

\medskip

We build the covering~$\mathcal{V}$ by induction on the skeletons of the triangulation, an example is shown in Figure \ref{fig: covering associated to a triangulation}. The base case is easy, the disks~$V_i^0$ are balls containing~$\Delta_i^0$ and such that \smash[t]{$V_i^0\cap V_j^0\neq \emptyset$} whenever~$i\neq j$, they are colored in red in Figure \ref{fig: covering associated to a triangulation}. 

For the inductive step, let us assume that we have already constructed the disks \smash[t]{$(V_i^\ell)_{i\in I_\ell},\, \ell=0,1,\ldots, k-1$} such that
$$\smash[t]{\widetilde{\Delta}_i^k:=\Delta_i^k-\bigcup_{\ell \leq k-1} V^\ell}$$
is a contraction of \smash[t]{$\Delta_i^1$} (i.e. there exists a contraction $f_i : \Delta^1_i \to \widetilde{\Delta}^k_i$ of factor $0 \leq k < 1$), where \smash[t]{$V^{\ell}=\bigcup_{j \in I_\ell} V_j^\ell$}. Let \smash[t]{$\widehat{\Delta}_i^k$} be a small thickening of \smash[t]{$\widetilde{\Delta}_i^k$. Then~$V_i^k$} is defined as a tubular neighborhood of \smash[t]{$\widehat{\Delta}_i^k$}. If this tubular neighborhood is small enough the open sets will verify~$V_i^k \cap V_j^k=\emptyset$ for all~$i\neq j$. This finishes the proof by induction.

\medskip

Now that~$\mathcal{V}$ is defined we define \smash[t]{$\mathcal{V}'=(V_i^{'k})_{i\in I_k}$ and~$\mathcal{U}=(U_i^k)_{i\in I_k}$}, two other open coverings obtained by thickening the \smash[t]{$V_i^k$} such that the following conditions hold:
$$ \forall k\leq 2, i\in I_k,\, \overline{V_i^k}\subset V_i^{'k} \subset \overline{V_i^{'k}}\subset U_i^k \text{ and } V_i^{'k}\cap V_j^{'k}=U_i^k\cap U_j^k=\emptyset,$$
whenever~$i\neq j$.

\begin{figure}[!ht]

	\centering
	\begin{tikzpicture}[line cap=round,line join=round,x=1cm,y=1cm,scale=1]
		\clip(-6.579903518410136,-3.6515633565121264) rectangle (7.964865357968676,6.0666902873463515);
		\draw [line width=1pt] (-2.64,-0.98)-- (-0.9,3.58);
		\draw [line width=1pt] (-0.9,3.58)-- (3.76,-1.74);
		\draw [line width=1pt] (3.76,-1.74)-- (-2.64,-0.98);
		\draw [line width=1pt, color=red] (-0.9,3.58) circle (1.5811264121690247cm);
		\draw [line width=1pt, color=red] (-2.64,-0.98) circle (1.7658262059841887cm);
		\draw [line width=1pt, color=red] (3.76,-1.74) circle (1.6410621407310344cm);
		\draw [line width=1pt, color=green] (-1.7044709017764281,-0.7002145150047672)-- (-1.7960138992606711,-1.471102914872077);
		\draw [line width=1pt, color=green] (-1.7960138992606711,-1.471102914872077)-- (2.3567077378581818,-1.964238609279941);
		\draw [line width=1pt, color=green] (2.3567077378581818,-1.964238609279941)-- (2.448250735342425,-1.193350209412631);
		\draw [line width=1pt, color=green] (2.448250735342425,-1.193350209412631)-- (-1.7044709017764281,-0.7002145150047672);
		\draw [line width=1pt, color=green] (-0.4431991618853157,2.417733212444913)-- (0.19198930398611536,2.974120101723272);
		\draw [line width=1pt, color=green] (2.7307052578368416,-1.2056941508430437)-- (3.365893723708273,-0.6493072615646848);
		\draw [line width=1pt, color=green] (-0.4431991618853157,2.417733212444913)-- (2.7307052578368416,-1.2056941508430437);
		\draw [line width=1pt, color=green] (3.365893723708273,-0.6493072615646848)-- (0.19198930398611536,2.974120101723272);
		\draw [line width=1pt, color=green] (-2.6828185840440613,0.27813316726286486)-- (-1.7699439755852873,-0.07020056491219367);
		\draw [line width=1pt, color=green] (-1.6147385142340043,3.077239557109911)-- (-0.7018639057752298,2.728905824934852);
		\draw [line width=1pt, color=green] (-1.6147385142340043,3.077239557109911)-- (-2.6828185840440613,0.27813316726286486);
		\draw [line width=1pt, color=green] (-1.7699439755852873,-0.07020056491219367)-- (-0.701863905775228,2.728905824934852);
		\draw [line width=1pt, color=green] (-3.5959430874938523,-0.5697681573690828)-- (-2.932918429414962,0.01815637299208195);
		\draw [line width=1pt, color=green] (-2.6670524365913075,-2.121711316060441)-- (-1.7963734825665252,-1.749760263087838);
		\draw [line width=1pt, color=green] (-1.6136671969896583,4.26400257475467)-- (-1.851223054116573,3.311010982292524);
		\draw [line width=1pt, color=green] (3.7813810810797603,-0.6026888570941396)-- (4.671094827667649,-1.058941998502897);
		\draw [line width=1pt, color=green] (2.9589863836992905,-2.0046780913616087)-- (3.8912457421943403,-2.5733377830758792);
		\draw [line width=1pt, color=green] (-0.3347841466017391,4.289358664560159)-- (-0.055063701873138415,3.250240499716883);
		\draw [line width=1pt, color=green] (-1.851223054116573,3.311010982292524)-- (-4.854352333855643,4.059612497514902);
		\draw [line width=1pt, color=green] (-1.6136671969896583,4.26400257475467)-- (-4.848199066272632,5.070286698024477);
		\draw [line width=1pt, color=green] (-4.865687491157133,2.197812910745057)-- (-2.932918429414962,0.01815637299208195);
		\draw [line width=1pt, color=green] (-3.5959430874938523,-0.5697681573690828)-- (-4.87376380452535,0.871278440015578);
		\draw [line width=1pt, color=green] (-2.6670524365913075,-2.121711316060441)-- (-2.06883230399965,-3.5220506218879635);
		\draw [line width=1pt, color=green] (-1.7963734825665252,-1.749760263087838)-- (-1.0365723431311509,-3.528335309625367);
		\draw [line width=1pt, color=green] (5.863724534757921,4.843520392051895)-- (-0.055063701873138415,3.250240499716883);
		\draw [line width=1pt, color=green] (3.7813810810797603,-0.6026888570941396)-- (5.855187578138971,3.441325267389073);
		\draw [line width=1pt, color=green] (4.671094827667649,-1.058941998502897)-- (5.841686439635061,1.2237632681219095);
		\draw [line width=1pt, color=green] (3.8912457421943403,-2.5733377830758792)-- (3.292638699073541,-3.5546927588929758);
		\draw [line width=1pt, color=green] (2.018240992494049,-3.5469338687463727)-- (2.9589863836992905,-2.0046780913616087);
		\draw [line width=1pt, color=green] (5.8705205505457245,5.9597659851985725)-- (-0.3347841466017391,4.289358664560159);
		\draw [line width=1pt] (-2.64,-0.98)-- (-4.869725647841241,1.5345456753803177);
		\draw [line width=1pt] (-2.64,-0.98)-- (-1.5527023235654007,-3.525192965756665);
		\draw [line width=1pt] (3.76,-1.74)-- (2.6554398457837953,-3.5508133138196745);
		\draw [line width=1pt] (3.76,-1.74)-- (5.848437008887016,2.3325442677554955);
		\draw [line width=1pt] (-0.9,3.58)-- (5.867122542651822,5.401643188625233);
		\draw [line width=1pt] (-0.9,3.58)-- (-4.8512757000641376,4.564949597769689);
		\draw [line width=1pt, color=blue] (-0.8565695771444796,3.1968704193613005)-- (-2.2914,-0.732822);
		\draw [line width=1pt, color=blue] (-2.291479511137008,-0.7328261044136882)-- (3.1383501254483552,-1.40136368729657);
		\draw [line width=1pt, color=blue] (3.1383501254483552,-1.40136368729657)-- (-0.8565695771444796,3.1968704193613005);
		\draw [line width=1pt, color=blue] (-2.7480417628619036,-0.5045449785512407)-- (-4.867942771515415,1.8273831118973165);
		\draw [line width=1pt, color=blue] (-2.7480417628619036,-0.5045449785512407)-- (-1.2479086500515328,3.343622571701445);
		\draw [line width=1pt, color=blue] (-1.2479086500515328,3.343622571701445)-- (-4.8351834850328546,4.321970253969076);
		\draw [line width=1pt, color=blue] (-0.4978,3.4088)-- (3.6438,-1.2709);
		\draw [line width=1pt, color=blue] (3.643,-1.2709)-- (5.85161,2.85440);
		\draw [line width=1pt, color=blue] (-0.4978,3.40884)-- (5.8655,5.1372);
		\begin{scriptsize}
			\draw [fill=black] (-2.64,-0.98) circle (2pt);
			\draw[color=black] (-2.9,-1.3) node {$\Delta^0_0$};
			\draw [fill=black] (-0.9,3.58) circle (2pt);
			\draw[color=black] (-0.7750406036221804,3.9224782837097916) node {$\Delta^0_1$};
			\draw [fill=black] (3.76,-1.74) circle (2pt);
			\draw[color=black] (4.2,-1.7) node {$\Delta^0_2$};
		\end{scriptsize}	
	
	\end{tikzpicture}
	\caption{The triangulation~$T$ is represented in black, the boundary of the open sets of the form~$V_i^0$ are painted in red, the boundary of the open sets of the form~$V_i^1$ are painted in green and the boundary of the open sets of the form~$V_i^2$ are painted in blue.}
	\label{fig: covering associated to a triangulation}
\end{figure}

Furthermore, if~$T$ is chosen such that for any simplex~$\sigma$ in~$T$, its \textit{star} (the union of all of the simplices touching~$\sigma$) is inside an open set of the covering~$\mathcal{W}$, then the open coverings~$\mathcal{V}$,~$\mathcal{V'}$ and~$\mathcal{U}$ are refinement of~$\mathcal{W}$. We fix now, and for the rest of the paper,~$T=(\Delta_i^k)_{i \in I_k}$ a triangulation with the above property along with~$\mathcal{V}=(V_i^k)_{i\in I_k}$,~$\mathcal{V'}=(V_i^{'k})_{i\in I_k}$ and~$\mathcal{U}=(U_i^k)_{i\in I_k}$ the open coverings associated to~$T$.


\subsection{Two corollaries of the area-preserving extension lemma for the annulus}\label{subsec: 2 coro}

As announced before, in this section we will prove two corollaries of Lemma \ref{lem: new version of the extension lemma} that will suit better our needs for the proof of the~$C^0$-fragmentation property. The corollaries are improved versions of Lemmas 2 and 3 in \cite[section 1.6.1]{EPP} and we will mimic their proofs.

\begin{coro}[Area-preserving extension lemma for disks]\label{coro: area-preserving extension lemma for disks}
Let~$D_1\subset D_2 \subset D_3 \subset \mathbb{R}^2$ be closed disks such that~$D_1\subset \text{Interior} (D_2)\subset D_2 \subset \text{Interior}(D_3)$. Let~$\phi: D_2 \rightarrow D_3$ be a smooth area-preserving embedding. If~$\phi$ is sufficiently~$C^0$-small, then there exists~$\psi \in \Ham(D_3)$ such that 

$$\psi\vert _{D_1}=\phi\vert_{D_1} \text{ and } \Vert\psi\Vert_{C^0}\leq C \Vert\phi\Vert_{C^0},$$
for some constant~$C>0$.
\end{coro}

\begin{rk}
The Corollary \ref{coro: area-preserving extension lemma for disks} can be completely adapted in the continuous setting if we ask to extend~$\phi$, a continuous measure-preserving embedding, by~$\psi$, an element of~$\Ker(\theta)$. We can then prove it by adapting the proof and using Lemma \ref{lem: area-preserving extension lemma for homeo} instead of Lemma \ref{lem: new version of the extension lemma}.
\end{rk}

\begin{proof}[Proof of Corollary \ref{coro: area-preserving extension lemma for disks}]
We mimic the proof in \cite{EPP}, adding only the sharper estimate of Lemma \ref{lem: new version of the extension lemma}.

Up to replacing~$D_2$ by a slightly smaller disk, we can assume that~$\phi$ is defined in a neighborhood of~$D_2$. Identify some small neighborhood of~$\partial D_2$ with~$\mathbb{A}_2$ so that~$\partial D_2$ is identified with~$S^1\times 0\subset \mathbb{A}_1 \subset \mathbb{A}_2$. If~$\phi$ is~$C^0$-small enough, we have~$\phi(\mathbb{A}_1)\subset \text{Interior}(\mathbb{A}_2)\subset \text{Interior}(D_3)\backslash \phi(D_1)$. Denote~$\delta: = \Vert\phi\Vert_{C^0}$.

Apply Lemma \ref{lem: new version of the extension lemma} and find~$h\in \text{Ham}(\mathbb{A}_2)$,

$$\Vert h\Vert_{C^0}\leq C \Vert\phi\Vert_{C^0},$$
for some constant~$C>0$ and so that~$h\vert_{\mathbb{A}_{1-D\delta}}=\phi$. Set~$\phi_1:=h^{-1}\circ \phi \in \text{Ham}(D_3)$. Note that~$\phi_1\vert_{D_1}=\phi$ and~$\phi_1$ is the identity on~$\mathbb{A}_{1-D\delta}$. Therefore we can extend~$\phi_1\vert_{D_2\cup \mathbb{A}_1}$ to~$D_3$ by the identity and get the required~$\psi$.
\end{proof}

\begin{coro}[Area-preserving extension lemma for rectangles]\label{coro: area-preserving extension lemma for rectangles}

Let~$\Pi_3=[0,R]\times [-c_3,c_3]$ be a rectangle and let~$\Pi_1\subset \Pi_2\subset\Pi_3$ be two smaller rectangles of the form~$\Pi_i=[0,R]\times [-c_i,c_i], (i=1,2)$,~$0<c_1<c_2<c_3$. Let~$\phi: \Pi_2 \rightarrow \Pi_3$ be a smooth area-preserving embedding such that 

\begin{itemize}
    \item~$\phi$ is the identity near~$0\times[-c_2,c_2]$ and~$R\times [-c_2,c_2]$,
    \item The area in~$\Pi_3$ bounded by the curve~$[0,R]\times y$ and its image under~$\phi$ is zero for some (and hence for all)~$y\in[-c_2,c_2]$.
\end{itemize}

If~$\phi$ is sufficiently~$C^0$-small, then there exists~$\psi \in \text{Ham}(\Pi_3)$ such that 

$$\psi \vert _{D_1}=\phi\vert_{D_1} \text{ and } \Vert\psi\Vert_{C^0}\leq C \Vert\phi \Vert_{C^0},$$
for some constant~$C>0$.

\end{coro}

\begin{rk}
Again Corollary \ref{coro: area-preserving extension lemma for rectangles} transposes completely in the continuous setting by taking~$\phi$ being a continuous measure-preserving embedding and extending it by~$\psi \in \Ker(\theta)$.
\end{rk}

\begin{proof}[Proof of Corollary \ref{coro: area-preserving extension lemma for rectangles}]
The proof relies on the last assumption of Lemma \ref{lem: new version of the extension lemma}, indeed we can identify the rectangle~$\Pi_3$ with the subset $I \times [-2,2] \subset \mathbb{A}_2$.
\end{proof}


\subsection{Definition of two obstructions}\label{subsec: def of the obstr}

In this section let~$(\Sigma,\omega)$ be a symplectic surface and~$T$,~$\mathcal{U}$,~$\mathcal{V}$ and~$\mathcal{V'}$ be a triangulation and 3 open coverings associated to it as defined in Section \ref{subsec: def of covering}. We will define the flux for an annulus, we call this quantity
$$\mathcal{O}: \Sympoc(\mathbb{A}_1)\rightarrow \mathbb{R}.$$ 
Thus,~$\mathcal{O}$ represents the obstruction for a diffeomorphism~$\phi \in \Sympc(\mathbb{A}_1)$ to belong to~$\Ham(\mathbb{A}_1)$. We also define~$\mathcal{A}_{i_1,i_2}(\phi)$ for~$i_1$ and~$i_2$ in~$I_0$ which will represent the obstruction of the extension of a certain embedding of~$V_j^{1}$ in~$V_j^{'1}$ if the two vertices~$\Delta_{i_1}^0$ and~$\Delta_{i_2}^0$ are joined by an edge~$\Delta_j^1$ (see Corollary \ref{coro: area-preserving extension lemma for rectangles}). The obstruction $\mathcal{A}$ can be seen as a way to define $\mathcal{O}$ on a general symplectic surface. 

\begin{dfn}[Definition of~$\mathcal{O}$]\label{def: definition of O}
Let~$\phi \in \Sympoc(\mathbb{A}_1)$,~$A \in S^1\times \{-1\}$ and~$B \in S^1\times \{1\}$ two points on the boundary of~$\mathbb{A}_1$ as in Figure \ref{fig: obstruction O} and~$\gamma=[0,1]\rightarrow \mathbb{A}_1, \, t\mapsto \gamma(t)$ an arc with endpoints~$\gamma(0)=A$ and~$\gamma(1)=B$. Then let~$h:[0,1]_s\times [0,1]_t\rightarrow \mathbb{A}_1$ be a smooth homotopy with fixed endpoints from~$\gamma$ to~$\phi(\gamma)$ such that~$h_{0,t}=\gamma(t)$ and~$h_{1,t}=\phi(\gamma(t))$. We can then define the quantity:
$$\mathcal{O}(\phi):=\int_0^1\int_0^1\omega(\partial_sh_{s,t},\partial_t h_{s,t}) \, ds \, dt.$$
\end{dfn}

\begin{rk}\label{rk: def of O}
We can also define~$\mathcal{O}$ in a more general setting by defining it as the area of a 2-cell defined by a continuous homotopy~$h$ and consider~$\omega$ as a measure, and even replace it by any Oxtoby-Ulam measure (see Section \ref{subsubsec: C^0 frag for homeo}). This will be useful for Theorem \ref{thm: fragmentation lemma for Homeo}.
\end{rk}

\begin{figure}
    \centering
    \includegraphics[scale=0.55]{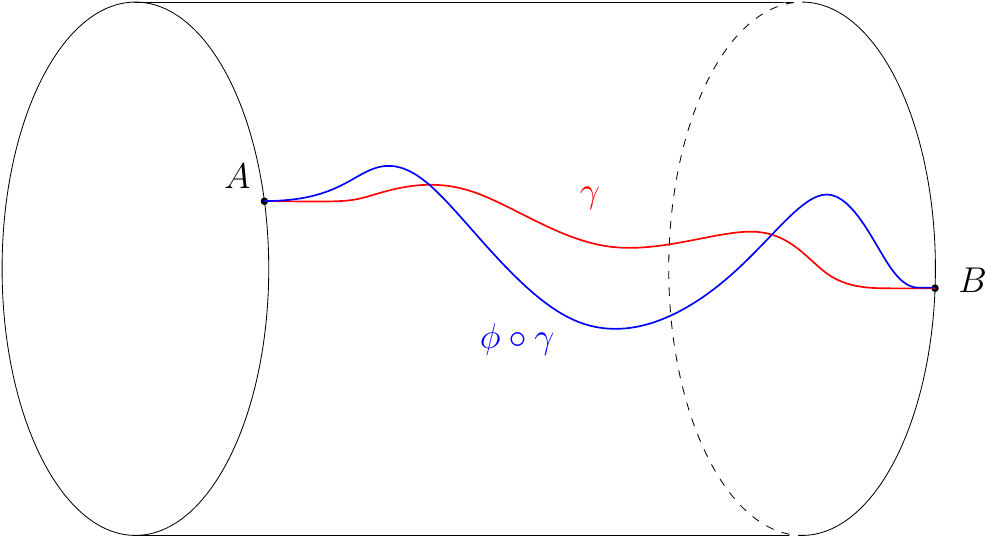}
    \caption{The obstruction $\mathcal{O}$ is the signed area between the red curve and the blue curve.}
    \label{fig: obstruction O}
\end{figure}

The next proposition gives some properties of~$\mathcal{O}$, those properties can also be proven in a continuous setting. This proposition is standard, and we omit the proof

\begin{prop}\label{prop: Properties of O}

The number $\mathcal{O}$ is well-defined, i.e. it does not depend on the choices of~$h$,~$\gamma$ nor on the choice of the points~$A$ and~$B$.

Moreover,~$\mathcal{O}$ is exactly the obstruction for~$\phi\in \Sympoc(\mathbb{A}_1)$ to be in~$\Ham(\mathbb{A}_1)$, that is~$\phi \in \Ham(\mathbb{A}_1)$ if and only if~$\mathcal{O}(\phi)=0$.
\end{prop}

We now state Lemma \ref{lem: surjectivity of O}, it proves the surjectivity of this obstruction. We also give a bound on the norm of a well-chosen preimage of a real number (it will be useful in the proof of Lemma \ref{lem: fragmentation on the 0-skeleton} for example). We will prove Lemma \ref{lem: surjectivity of O} at the end of this proof.

\begin{lemma}[Surjectivity of~$\mathcal{O}$ and an estimate on the norm of a pre-image]\label{lem: surjectivity of O}
With the previous definitions, the obstruction~$\mathcal{O}: \Sympc(\mathbb{A}_1)\rightarrow\mathbb{R}, \, \psi \mapsto \mathcal{O}(\psi)$ is a surjective function. 
	
Moreover, there exists a constant~$C>0$ such that for all~$\varepsilon \in \mathbb{R}$, there exists~$\psi_{\varepsilon}\in \Sympc(\mathbb{A}_1)$ such that~$\mathcal{O}(\psi_{\varepsilon})=\varepsilon$ and~$$\Vert \psi_{\varepsilon} \Vert_{C^0}\leq C\vert \varepsilon\vert.$$
\end{lemma}

\begin{proof}[Proof of Lemma \ref{lem: surjectivity of O}]
Let~$\varepsilon\in \mathbb{R}$ and define~$H_s^\varepsilon(x,y)= \chi(y)\varepsilon$ a Hamiltonian-like function, where~$\chi$ is a smooth function in~$[-1,1]$, satisfying~$\chi(-1)=0$ and~$\chi(1)=1$ and~$\chi'$ is compactly supported in~$[-1,1]$. We let~$C=\Vert\chi'\Vert_0$. We then define~$\phi_s$ as the flow generated by this Hamiltonian-like function. That is, define the vector field~$X_s$ by the equation
$$\iota_{X_s}\omega=dH_s^\varepsilon=\chi'(y)\varepsilon \,dy.$$

So~$X_s=\chi'(y)\varepsilon \, \frac{\partial}{\partial x}$ and~$\phi_s(x,y)=(x+\varepsilon \chi'(y) s, y)$, where~$x$ is in~$\mathbb{R}/\mathbb{Z}$ and~$x+\varepsilon \chi'(y) s$ is taken in~$\mathbb{R}/\mathbb{Z}$. Then~$\Vert\phi\Vert_{C^0}\leq C \vert\varepsilon\vert$ and~$\mathcal{O}(\phi)=\int_0^1 H_s(1)\, ds=\varepsilon$ as wanted.
\end{proof}

As seen in Remark \ref{rk: def of O} there is an analogous definition of~$\mathcal{O}$ in the continuous setting, it is not hard to see that Proposition \ref{prop: Properties of O} transposes in the continuous setting once we work with~$\Ker(\theta)$ instead of~$\Ham(\mathbb{A}_1)$.

\begin{dfn}[Definition of~$\mathcal{A}$]\label{def: def of A}
Let~$T$,~$\mathcal{U}$ and~$\mathcal{V}$ be as in Section \ref{subsec: def of covering}. Let~$\phi \in \Sympo(\Sigma)$ be $C^0$-small such that for all~$i\in I_0$,~$\phi\vert_{ V_i^{0}}=Id$. Given~$\Delta_{i_1}^0$ and~$\Delta_{i_2}^0$ two vertices in~$T$ linked by an edge~$\Delta_j^1$ parameterized by an arc~$\gamma$ with~$\gamma(0)=\Delta_{i_1}^0$ and~$\gamma(1)=\Delta_{i_2}^0$, if $\phi$ is $C^0$-small enough then $\phi(\Delta^1_j)\subset  U^1_j$. Then, we define:
$$\mathcal{A}_{i_1,i_2}(\phi)=\int_0^1\int_0^1 \omega(\partial_s h_{s,t},\partial_t h_{s,t}),$$
where~$h: [0,1]\times [0,1]\rightarrow U_j^1$ is a homotopy with fixed endpoints from~$\gamma$ to~$\phi(\gamma)$.
\end{dfn}

\begin{rk}\label{rk: def of A}
We can also define~$\mathcal{A}$ via integration on 2-cells instead. Again it allows us to define~$\mathcal{A}$ on the continuous case.
\end{rk}

We prove now several properties of this obstruction~$\mathcal{A}$.

\begin{prop}\label{prop: properties of A}
Let~$T$,~$\mathcal{U}$ and~$\mathcal{V}$ be given as in Section \ref{subsec: def of covering}. Let $\phi$ be a symplectic diffeomorphism of $\Sigma$ connected to the identity. We assume that for all pairs of indices~$i_1$ and~$i_2$ in~$I_0$ such that~$\Delta_{i_1}^0$ and~$\Delta_{i_2}^0$, connected by an edge~$\Delta_j^1$,~$\phi$ is~$C^0$-small enough such that $\phi(\Delta^1_j)\subset  U^1_j$. Moreover $\phi$ verifies that, for all~$i\in I_0$,~$\phi\vert_{ V_i^{0}}=Id$. We have the following four properties:
\begin{itemize}
    \item[(i)]The quantity~$\mathcal{A}_{i_1,i_2}(\phi)$ is well defined, i.e. it does not depend on the choice made in its definition.
    \item[(ii)] The following identity holds:~$$\mathcal{A}_{i_1,i_2}(\phi)=-\mathcal{A}_{i_2,i_1}(\phi).$$
    \item[(iii)] There exists a constant~$C>0$ that does not depend on~$\phi$ such that, 
   ~$$\vert \mathcal{A}_{i_1,i_2}(\phi)\vert\leq C \Vert \phi \Vert_{C^0}.$$
    \item[(iv)] If~$\phi\in \Ham(\Sigma)$, then for a loop of the triangulation i.e. a set of indices~$i_1, i_2, \ldots, i_m$ such that~$i_p$ and~$i_{p+1}$ are linked by an edge in~$T$ (with the convention that~$i_{m+1}=i_1$), we have: 
   ~$$\sum_{p=1}^m \mathcal{A}_{i_p,i_{p+1}}(\phi)=0.$$
\end{itemize}
\end{prop}

\begin{proof}

We are going to prove the claims in the order they appear.

\begin{itemize}
    \item[(i)] Since in Definition \ref{def: def of A}, the homotopy $h$ is inside the disk $U^1_j$ and since the disk is contractible the space of homotopies $h$ is simply connected and thus the choice of homotopy does not change the value of the integral.
    
    \item[(ii)] If~$h_{s,t}$ is an isotopy from~$\gamma$ to~$\phi(\gamma(t))$, then~$h_{s,1-t}$ is an isotopy from~$\gamma(1-t)$ to~$\phi(\gamma(1-t))$, so after a change of variable in the integral we have the identity:
    $$\mathcal{A}_{i_1,i_2}(\phi)=\int_0^1\int_0^1 \omega(\partial_s h_{s,t},\partial_t h_{s,t})=-\mathcal{A}_{i_2,i_1}(\phi).$$
    
    \item[(iii)] We define the tubular neighborhood $N$ of $\Delta^1_j$ with width $2\Vert \phi \Vert_{C^0}$, since $\phi$ is the identity on $V^0_{i_1}$ and on $V^0_{i_2}$ the image $\phi(\Delta^1_j)$ does not cross the end of the edge $\Delta^1_j$. This implies that the area below $\phi(\Delta^1_j)$ in $N$ is smaller than $2\Vert \phi \Vert_{C^0}\ell$, where $\ell$ denotes the length of $\Delta^1_j$. Hence,
    \[\vert \mathcal{A}_{i_1,i_2}(\phi)\vert\leq C \Vert \phi \Vert_{C^0}.\]
    for some constant $C > 0$.
    
    \item[(iv)] Let~$\gamma$ be the piece-wise smooth path going through \smash[t]{$\Delta_{j_1}^{1},\Delta_{j_2}^1,\ldots, \Delta_{j_m}^1$}, where \smash[t]{$\Delta_{j_p}^1$} links \smash[t]{$\Delta_{i_p}^0$ to~$\Delta_{i_{p+1}}^0$}, the path should not go twice through the same vertex. Let~$\phi_t$ be a Hamiltonian isotopy from~$Id$ to~$\phi$, then \smash[t]{$\widetilde{\text{Flux}}(\phi_t)=0$}, this means that the area of the cylinder~$\phi_t(\gamma)$ is zero. However, nothing tells us a priori that the area of the cylinder of~$\phi_t(\gamma)$ is the same as the sum~$\sum_{p=1}^m \mathcal{A}_{i_p,i_{p+1}}(\phi)$, obtained also as the area of some cylinder between~$\gamma$ and~$\phi\circ \gamma$ but with support in \smash[t]{$\bigcup U_{j_p}^1$}. However, we can glue those two cylinders to obtain one closed 2-cycle~$\sigma_2$ so
    $$\int_{\sigma_2}\omega=\widetilde{\text{Flux}}(\phi_t)(\gamma)+\sum_{p=1}^m \mathcal{A}_{i_p,i_{p+1}}(\phi)=\sum_{p=1}^m \mathcal{A}_{i_p,i_{p+1}}(\phi)$$ 
    has value in~$\omega \cdot H_2(\Sigma,\mathbb{Z})= <\omega, [\Sigma]>\mathbb{Z} \subset \mathbb{R}$. If~$\phi$ is~$C^0$-small enough then by the third point of the proposition,~$\vert\sum_{p=1}^m \mathcal{A}_{i_p,i_{p+1}}(\phi)\vert\leq m C \Vert \phi \Vert_{C^0}\leq \vert I_k\vert \Vert \phi \Vert_{C^0}$ is inside this subgroup so must be 0 for~$\Vert \phi \Vert_{C^0}$ small enough. This finishes the proof of Proposition \ref{prop: properties of A}.
\end{itemize}
\end{proof}


\subsection{Proof of Theorems \ref{thm: fragmentation lemma on surfaces with Lipschitz estimate} and \ref{thm: fragmentation lemma for Homeo}}\label{subsec: pf of the thm}

We have now all the tools to prove Theorem \ref{thm: fragmentation lemma on surfaces with Lipschitz estimate} and Theorem \ref{thm: fragmentation lemma for Homeo}.

\begin{proof}[Proof of Theorem \ref{thm: fragmentation lemma on surfaces with Lipschitz estimate}]

Let~$T$ be a triangulation such that the star of every vertices of~$T$ are included in one of the open sets of the subcovering of~$\mathcal{W}$, we will say that such triangulation is \textit{good}. We will consider the open coverings~$\mathcal{V},\, \mathcal{V}'$ and~$\mathcal{U}$ associated to~$T$. Then the three coverings~$\mathcal{V}$,~$\mathcal{V}'$ and~$\mathcal{U}$ are finer than~$\mathcal{W}$. We will prove the fragmentation theorem on~$\mathcal{U}$ which will imply it for~$\mathcal{W}$ by a simple argument described at the end of the proof. We recall that for all~$k=0,1,2$ and~$i\in I_k$, \smash[t]{$V_i^k\subset V'^k_i\subset U_i^k$}.

Let us describe briefly an outline of the proof. In order to fragment a diffeomorphism~$\phi$ we will proceed in 3 steps.

In Lemma \ref{lem: fragmentation on the 0-skeleton} and Lemma \ref{sublemma: naive fragmentation on the 0-skeleton} we will start the fragmentation by finding Hamiltonian diffeomorphisms compactly supported in~$(U_i^{0})$, agreeing with~$\phi$ on~$V_i^0$ and some additional condition that will be needed in order to prove Lemma \ref{lem: fragmentation on the 1-skeleton}. We will then define~$\phi'$ as the Hamiltonian diffeomorphism~$\phi$ were we pre-composed by the inverses of the diffeomorphisms we just constructed. Thus,~$\phi'$ is a Hamiltonian diffeomorphism which is the identity on the sets~$(U_i^{0})$. 

In Lemma \ref{lem: fragmentation on the 1-skeleton} we find Hamiltonian diffeomorphisms compactly supported in~$U_i^1$ and agreeing with~$\phi'$ on~$V_i^1$. With the same construction as above we define~$\phi''$.

In Lemma \ref{lem: fragmentation on the 2-skeleton} we will finish the fragmentation without any difficulty since~$\phi''$ is actually naturally fragmented.

\begin{lemma}[Fragmentation on the 0-skeleton]\label{lem: fragmentation on the 0-skeleton}
Let~$T=(\Delta_i^k)_{i\in I_k}$ be a good triangulation~$\mathcal{U}$,~$\mathcal{V}$ and~$\mathcal{V}'$ open coverings associated with~$T$ as described in Section \ref{subsec: def of covering}. Let~$\phi\in \Ham(\Sigma)$ be a~$C^0$-small diffeomorphism. Then we can find the following~$C^0$-fragmentation:
$$\phi=\phi_1^{(0)} \circ \phi_2^{(0)} \cdots \circ 
\phi_{\ell}^{(0)}\circ \phi',$$
where~$\ell:=\vert I_0 \vert$, for all~$i \in I_0$,~$\phi_i^{(0)}\in\Ham(U_i^0)$, verifies~$\phi_i^{(0)}\vert_{V^0_i}=\phi$ and satisfies the estimate 
$$\Vert \phi_p^{(0)} \Vert_{C^0}\leq C \Vert \phi \Vert_{C^0},$$
where~$C>0$ is a constant.

Moreover,~$\mathcal{A}_{i,j}(\phi')=0$ for all~$i,j \in I_0$ linked by an edge in~$T$.
\end{lemma}

The first step in order to prove Lemma \ref{lem: fragmentation on the 0-skeleton} is to prove Lemma \ref{sublemma: naive fragmentation on the 0-skeleton}, it is a fragmentation on the open sets on the vertices but we do not ask for the obstruction~$\mathcal{A}$ to be 0.

\begin{lemma}\label{sublemma: naive fragmentation on the 0-skeleton}
Let~$T=(\Delta_i^k)_{i\in I_k}$ be a good triangulation~$\mathcal{U}$,~$\mathcal{V}$ and~$\mathcal{V}'$ open coverings associated with T as described in Section \ref{subsec: def of covering}. Let~$\phi\in \Ham(\Sigma)$ a~$C^0$-small diffeomorphism, then we can find the following~$C^0$-fragmentation:
$$\phi=\check{\phi}_1 \circ \check\phi_2 \cdots \circ 
\check\phi_{\ell}\circ \widetilde{\phi},$$
where~$\ell=\vert I_0 \vert$, for all~$i \in I_0$,~$\check\phi_i\in \Ham(U_i^0)$, verifies~$\check\phi_i\vert_{V^0_i}=\phi$ and satisfies the estimate 
$$\Vert \check \phi_i \Vert_{C^0}\leq C \Vert \phi \Vert_{C^0},$$
where~$C>0$ is a constant depending on the choice of~$T$ and the open coverings associated to it. The Hamiltonian diffeomorphism \smash[t]{$\widetilde{\phi}$} is then supported in~$\Sigma \backslash V^0$ and its~$C^0$-norm satisfy a Lipschitz bound in the norm~$\Vert\phi \Vert_{C^0}$.
\end{lemma}

\begin{proof}
Let~$i \in I_0$, by Corollary \ref{coro: area-preserving extension lemma for disks} and since~$\phi$ is a~$C^0$-small diffeomorphism and an area-preserving embedding of~$V'^0_i$ in~$U_i^{0}$, there exists~$\check\phi_i\in \Ham(U_i^0)$ such that~$\check\phi_i\vert_{V_i}=\phi$. Moreover, there exists a constant~$D>0$ such that~$\Vert \check\phi_i \Vert_{C^0}\leq D \Vert \phi \Vert_{C^0}$. Then there exists a diffeomorphism~$\widetilde{\phi}$ such that
$$\phi=\bigcirc_{i\in I_k}\phi_i \circ \widetilde{\phi}$$ 
(there is no issue with the composition since the supports of the~$\phi_i$ are disjoint). We then have that~$\widetilde{\phi}$ is supported in~$\Sigma\backslash V^0$ and is a Hamiltonian diffeomorphism of~$\Sigma$. Also,~$\Vert \widetilde{\phi} \Vert_{C^0} \leq \Vert\phi\Vert_{C^0}+\sum \Vert \check \phi_i\Vert_{C^0}\leq C \Vert \phi\Vert_{C^0}$, for~$C>0$ a constant.
\end{proof}

It is time now to prove Lemma \ref{lem: fragmentation on the 0-skeleton}.

\begin{proof}[Proof of Lemma \ref{lem: fragmentation on the 0-skeleton}]

We apply first Lemma \ref{sublemma: naive fragmentation on the 0-skeleton} and we now want to do slight modifications on the diffeomorphisms~$\check \phi_p$ in order to vanish the obstruction~$\mathcal{A}$ on each edge.

In order to do this, we first take a fixed index~$i\in I_0$ and its corresponding vertex~$\Delta_i^0$. We will define for a vertex~$j\in I_0$ the real number~$C(j)$ by~$$C(j):=\sum_{p=0}^{m-1}\mathcal{A}_{i_p,i_{p+1}}(\widetilde\phi),$$
where~$i_0=i,i_1,\ldots,i_m=j$ is a sequence of indices such that they are all linked by an edge in~$T$. Using property (iv) of Proposition \ref{prop: properties of A} we see that this value does not depend on the sequence~$(i_p)$ we take. We define now~$\phi_j^{(0)}=\rho_{j}\circ \check \phi_j$ where~$\rho_j$ is a compactly supported diffeomorphism from the annulus~$U_j^0\backslash V_j^0$ to itself, and such that~$\mathcal{O}(\rho_j)=C(j)$ (we identify~$\partial V_j^0$ with~$S^1 \times \{-1\}$ to match Definition \ref{def: definition of O}), note that Lemma \ref{lem: surjectivity of O} allows us to take~$\rho_j$ with a Lipschitz estimate. Then the diffeomorphisms~$\phi_j^{(0)}$ are the fragments needed in Lemma \ref{lem: fragmentation on the 0-skeleton}, and~$$\phi'=\bigcirc_{j} (\rho_j)^{-1}\widetilde{\phi}$$ satisfy all the conditions of the conclusion of Lemma \ref{lem: fragmentation on the 0-skeleton}.

\end{proof}

We now work on~$\phi'$ and fragment it in Lemma \ref{lem: fragmentation on the 1-skeleton}.

\begin{lemma}[Fragmentation on the 1-skeleton]\label{lem: fragmentation on the 1-skeleton}
Let~$\phi'$ be the resulting Hamiltonian diffeomorphism after applying Lemma \ref{lem: fragmentation on the 0-skeleton}, then there exists a fragmentation of~$\phi'$,
$$\phi'=\phi_1^{(1)}\circ \phi_2^{(1)}\cdots \circ \phi_m^{(1)}\circ \phi'',$$
where~$m=\vert I_1 \vert$, for all~$i \in I_1$,~$\phi_i^{(1)} \in \Ham(U_i^1)$,~$\phi_i^{(1)}\vert_{V_i^1}=\phi'$ and the following estimate is true
$$\Vert \phi_i^{(1)} \Vert_{C^0}\leq C \Vert \phi' \Vert_{C^0},$$
where~$C>0$ is a constant (that depends on~$T$). The resulting~$\phi''$ is then supported in~$\Sigma \backslash (V^0 \cup V^1)$ and satisfy a Lipschitz estimate with respect to~$\Vert\phi' \Vert_{C^0}$ and thus with respect to~$\Vert\phi \Vert_{C^0}$.
\end{lemma}

\begin{proof}
Let~$i\in I_1$, and~$i_1$ and~$i_2$ are the vertices of the edge~$\Delta_i^1$. Then, since~$\phi'$ is~$C^0$-small~$\phi'$ is an area-preserving embedding of~$V'^1_i$ in~$U_i^{1}$ being equal to the identity on~$V_i^1\cap V_{i_1}^0$ and~$V_i^1 \cap V_{i_2}^0$, also the condition~$\mathcal{A}_{i_1,i_2}(\phi')=
0$ is exactly the condition we need to apply Corollary \ref{coro: area-preserving extension lemma for rectangles}. We have now \smash[t]{$\phi_i^{(1)} \in \Ham(U_i^1)$} such that \smash[t]{$\phi_i^{(1)}\vert_{V_i^1}=\phi'$ and~$\Vert\phi_i^{(1)}\Vert_{C^0}\leq D \Vert \phi' \Vert_{C^0}$.} So there exists a diffeomorphism~$\phi''$ such that \smash[t]{$\phi'=\bigcirc_{i\in I_1} \phi_i^{(1)}\circ \phi''$}. Then~$\phi''$ is a Hamiltonian diffeomorphism of~$\Sigma$ itself, it is compactly supported in~$\Sigma \backslash (V_0 \cup V_1)$ and satisfy~$\Vert \phi'' \Vert_{C^0}\leq \Vert \phi' \Vert_{C^0}+\sum_{i\in I_1} \Vert\phi_i^{(1)}\Vert_{C^0}\leq C \Vert\phi \Vert_{C^0}$.
\end{proof}

\begin{lemma}[Fragmentation on the 2-skeleton]\label{lem: fragmentation on the 2-skeleton}
Let~$\phi''$ be the resulting Hamiltonian diffeomorphism from Lemma \ref{lem: fragmentation on the 1-skeleton}, we can fragment it
$$\phi''=\phi_1^{(2)}\circ \phi_2^{(2)}\cdots \circ \phi_n^{(2)},$$
where~$n=\vert I_2 \vert$, for all~$i \in I_2$,~$\phi_i^{(2)}\in \Ham(U_i^2)$ and~$$\Vert \phi_i^{(2)} \Vert_{C^0}\leq C \Vert \phi'' \Vert_{C^0},$$
where~$C>0$ is a constant.

\end{lemma}

\begin{proof}
$\phi''$ has now support in~$\Sigma \backslash (V^0 \cup V^1)\subset \bigcup U_i^{2}$. So it can be decomposed in Hamiltonian diffeomorphism with support in the~$U_i^2$ (pairwise disjoint) and the bound is also immediate.
\end{proof}

Combining Lemmas \ref{lem: fragmentation on the 0-skeleton}, \ref{lem: fragmentation on the 1-skeleton} and \ref{lem: fragmentation on the 2-skeleton} we obtain a fragmentation with a finitely bounded number of fragments. We describe now the procedure to give the result in the shape of Theorem \ref{thm: fragmentation lemma on surfaces with Lipschitz estimate}.

We want to swap the support of the diffeomorphisms, let~$\psi=fg$ for two diffeomorphisms~$f$ and~$g$ such that~$\text{supp}(f)\subset B_1'\subset B_1$ and~$\text{supp}(g)\subset B_2$. Then~$\psi=g(g^{-1}fg)$ and we want to show now that~$g^{-1}fg$ is supported inside~$B_1''$ the set of points whose distance to~$B_1'$ is smaller than~$\Vert g \Vert_{C^0}$ if~$f$ and~$g$ have a small~$C^0$-norm. Indeed, if~$d(x,B_1')\geq \Vert g\Vert_{C^0}$ then~$g^{-1}(f(g(x)))=g^{-1}(g(x))=x$. This procedure is the kind of commutation we needed
$$\psi=\underbrace{f}_{\text{support in } B_1}\underbrace{g}_{\text{support in } B_2}=\underbrace{g}_{\text{support in } B_2}\underbrace{g^{-1}fg}_{\text{support in } B_1}.$$

We need to show that we can apply this procedure repeatedly. Note that one needs to apply it only a finite number of time since there are at most~$N:=\ell+m+n$ fragments in the fragmentation we have for now and one can obtain any permutation with~$N$ elements with at most~$N!$ transpositions. If we have~$\Vert \phi \Vert_{C^0}$ small enough, then the successive thickenings of~$B_1'$ in the operation described above will stay inside~$B_1$. We thus obtain a fragmentation~$\phi=\phi_1 \phi_2\cdots \phi_m$ with~$\phi_i\in \Ham(W_i)$ and~$\Vert\phi_i\Vert_{C^0} \leq C \Vert\phi\Vert_{C^0}$ for some constant~$C>0$.

\end{proof}

\begin{proof}[Proof of Theorem \ref{thm: fragmentation lemma for Homeo}]
The proof transposes for~$\text{Ker}(\theta)$ by adapting directly Corollaries \ref{coro: area-preserving extension lemma for disks} and \ref{coro: area-preserving extension lemma for rectangles}, then the obstruction~$\mathcal{O}$ and~$\mathcal{A}$ and finally Lemma \ref{lem: fragmentation on the 0-skeleton}, \ref{sublemma: naive fragmentation on the 0-skeleton}, \ref{lem: fragmentation on the 1-skeleton} and \ref{lem: fragmentation on the 2-skeleton}.
\end{proof}


\section{Proof of the area-preserving extension lemma}\label{sec: pf of the area-preserving extension lemma}


\subsection{Preliminaries}\label{subsec: preliminary}

We need to state three propositions about the area forms on a surface in order to carry out the proof of Lemma \ref{lem: new version of the extension lemma}. The two first propositions are already well-known. The third one is new and is the key piece that allows us to go from a Hölder estimate for the area-preserving extension lemma to the Lipschitz estimate of Lemma \ref{lem: new version of the extension lemma}.

We recall that a Borel measure~$\mu$ on a compact manifold~$X$ is said to be an \textit{OU (Oxtoby-Ulam) measure} if it is nonatomic, of full support and is zero on the boundary. The first proposition is proven in \cite{OU}.

\begin{prop}\label{prop: OU proposition}
Let~$\mu$ and~$\nu$ be two OU measures on a rectangular~$r$-cell~$R$ such that~$\mu(R)=\nu(R)$, then there exists a homeomorphism~$h$ which restricts to the identity on the boundary of~$B$ such that~$h^*\nu=\mu$.
\end{prop}

We recall here some consequences of Moser's trick \cite{Moser} as described in \cite[Section 3.4.1]{Seyfaddini}.

\begin{prop}[Moser's trick]\label{prop: Moser's trick}
Let~$M$ be a compact connected oriented manifold of dimension~$n$, possibly with a non-empty boundary~$\partial M$, and let~$\omega_1$,~$\omega_2$ be two volume forms on~$M$. Assume that~$\int_M \omega_1=\int_M\omega_2$. If~$\partial M\neq 0$, we also assume that the forms~$\omega_1$ and~$\omega_2$ coincide on~$\partial M$.

Then there exists a diffeomorphism~$f: M \rightarrow M$, isotopic to the identity, such that~$f^* \omega_2=\omega_1$. Moreover,~$f$ can be chosen to satisfy the following properties: 

\medskip

(i) If~$\partial M \neq \emptyset$, then~$f$ is the identity on~$\partial M$, and if~$\omega_1$ and~$\omega_2$ coincide near~$\partial M$, then~$f$ is the identity near~$\partial M$.

\medskip
 
(ii) If~$M$ is partitioned into polyhedra (with piece-wise smooth boundaries), so that~$\omega_1-\omega_2$ is zero on the~$(n-1)$-skeleton~$\Gamma$ of the partition and the integral of~$\omega_1$ and~$\omega_2$ on each of the polyhedra are equal, then~$f$ can be chosen to be the identity on~$\Gamma$.
 
\medskip
 
(iii) Suppose that~$\omega_2=\chi\omega_1$ for a function~$\chi$, then the diffeomorphism~$f$ can be chosen to satisfy the following estimate: 
$$\Vert f \Vert_{C^0}\leq C \lVert\chi-1 \rVert_{C^0},$$
for some~$C>0$. Here,~$\lVert \cdot \rVert_{C^0}$ denotes the standard sup norm on functions.
\end{prop}

We describe now a lemma which allow us to adjust two volume forms by a~$C^0$-small diffeomorphism if they disagree on a small strip only. This is the new idea that allowed us to improve the continuity property for the extension lemma.

\begin{prop}\label{prop: adjusting volume forms}
Let~$C\in \mathbb{R}$ be a constant and~$M^{n-1}$ a compact manifold equipped with a volume form~$\omega'$ and a distance~$d'$, we also define~$d$ the product distance on~$M^{n-1}\times [-1,1]$. Let~$\omega=\omega' \wedge dz$ and~$\Omega$ two volume forms on~$M^{n-1}\times [-1,1]$, where~$z$ denote the coordinate on~$[-1,1]$. Let~$\chi$ be the function such that~$\omega=(1+\chi)\Omega$. We assume that: 

\begin{itemize}
    \item~$\int_{M^{n-1}\times [-1,1]}\omega =\int_{M^{n-1}\times [-1,1]}\Omega.$
    \item~$\Vert \chi \Vert_{C^0}\leq C.$
    \item There exists~$\delta>0$ such that~$\text{Supp}(\chi)\subset M^{n-1}\times [0,\delta].$
\end{itemize}
Then, there exists a constant~$D\in \mathbb{R}$ independent of~$\delta$ such that we can find~$f\in \text{Diff}(M^{n-1}\times[-1,1])$,~$f\equiv Id$ on~$M^{n-1}\times [-1,0]$,~$f^*\Omega=\omega$ and~$\Vert f \Vert_{C^0}\leq D\delta$.
\end{prop}

\begin{proof}
We will denote by~$z$ the last coordinate of the manifold~$M^{n-1}\times [-1,1]$ and by~$x$ the coordinate on~$M^{n-1}$. Let us describe two diffeomorphisms of~$N^n:=M^{n-1}\times \mathbb{R}$, which will be used in the definition of the desired map $f$ on~$M^{n-1}\times [-1,1]$.

Define~$\rho_\delta: \mathbb{R}\rightarrow \mathbb{R}, x\mapsto 2\delta x$. We can then define two maps from~$N^n$ to~$N^n$ by:
$$\Psi_1(x,z):=\left(x,\int_0^z(1+\chi(x,t))\,dt\right)$$
and 
$$\Psi_2(x,z):=\left(x,\int_0^z(1+\rho_\delta'(t)\chi(x,\rho_\delta(t)))\,dt\right).$$

We can then compute 
$$d\Psi_1(x,z):=\begin{pmatrix}
	Id & *\\
	0 &  1+\chi(x,z)
\end{pmatrix},$$
and
$$d\Psi_2(x,z):=\begin{pmatrix}
	Id & *\\
	0 &  1+\rho_\delta'\chi(x,\rho_\delta(z))
\end{pmatrix}.$$

Since~$1+\chi>0$ the function~$\Psi_1$ is a diffeomorphism, in the same manner~$\Psi_2$ is also a diffeomorphism. A simple computation also shows that~$(\Psi_{1})^*\omega=\text{det}(d \Psi_1)\omega=(1+\chi(x,z))\omega=\Omega$ and also~$(\Psi_{2})^*\omega=(1+\rho_\delta'\chi(x,\rho_\delta(z)))\omega$ which gives 
$$\left(((\Psi_1)^{-1} \circ \Psi_2)^*\Omega\right)_{(x,z)}=(1+\rho_\delta'(z)\chi(x,\rho_\delta(z)))\omega_{(x,z)}.$$

Moreover, 
$$d\left((x,z),\Psi_1(x,z)\right)=\left\vert\int_0^z \chi(x,\rho_\delta(t))\,dt\right\vert \leq \Vert \chi \Vert_{C^0} \delta\leq C \delta$$
so~$\Vert \Psi_1 \Vert_{C^0}\leq C\delta$ and similarly~$\Vert \Psi_2 \Vert_{C^0}\leq C \delta$.

If~$z>\delta$, since~$\chi$ has support in a strip we have~$\Psi_1(x,z)=(x,z+c(x))$, where~$c(x)$ is a function independent of~$z$ and whose value is~$c(x)=\int_0^\delta \chi(x,t)\, dt$. If~$z>0.5$, we have, by the same argument, that~$\Psi_2(x,z)=(x,z+d(x))$, where~$d(x)=\int_0^{0.5}\rho_\delta'(t)\chi(x,\rho_\delta(t))\, dt$. 

It follows that those two diffeomorphisms aren't compactly supported in~$M^{n-1}\times [-1,1]$, however since~$c(x)=d(x)$ (simple change of variable in the integral),~$\Psi:=(\Psi_1)^{-1} \circ \Psi_2$ can be restricted on~$M^{n-1}\times [-1,1]$ to a compactly supported diffeomorphism. Moreover,~$$\left(\Psi^*\Omega\right)_{(x,z)}=\left(((\Psi_1)^{-1} \circ \Psi_2)^*\Omega\right)_{(x,z)}=(1+\rho_\delta'(z)\chi(x,\rho_\delta(z)))\omega_{(x,z)}$$ and 
$$\Vert\Psi\Vert_{C^0}\leq \Vert \Psi_1\Vert_{C^0}+\Vert \Psi_2\Vert_{C^0}\leq 2C\delta.$$

By Proposition \ref{prop: Moser's trick} one can find a function~$h\in \text{Diff}_c(M^{n-1}\times[-1,1])$ such that~$h^*(\Psi^*(\Omega))=\omega$ and there exists a constant~$C'$ independent of~$\omega$ and~$\Omega$ such that~$\Vert h \Vert_{C^0}\leq C' \Vert \rho_\delta' \chi(\cdot,\rho_\delta(\cdot))\Vert_{C^0}\leq 2C'C\delta$. It then follows that~$\Psi h$ is the diffeomorphism we needed.

\end{proof}


\subsection{Proofs of Lemma \ref{lem: new version of the extension lemma} and Lemma \ref{lem: area-preserving extension lemma for homeo}}\label{subsec: pf of the area-preserving extension lemma}

The proof of Lemma \ref{lem: new version of the extension lemma} (resp. Lemma \ref{lem: area-preserving extension lemma for homeo}) will go as follows. We first extend the area-preserving embedding~$\phi$ to a~$C^0$-small diffeomorphism (resp. homeomorphism)~$f$ not necessarily preserving the area form. The diffeomorphism (resp. homeomorphism)~$f$ will verify that its~$C^0$-norm is smaller than~$C_1 \delta$ for some constant~$C_1>0$ independent of~$\delta$. Then working on the area form~$f^*\omega$, one can find~$g$ such that:

\begin{itemize}
	\item[$\bullet$] the~$C^0$-norm of~$g$ is bounded by~$C_2 \delta$ for some constant~$C_2>0$ independent of~$\delta$;
    \item[$\bullet$] we have~$g \in \text{Diff}_{0,c}(\mathbb{A}_{1+D_2 \delta})$ for some constant~$D_2$ independent of~$\delta$;
    \item[$\bullet$] if~$\chi$ is such that~$g^*\omega=(1+\chi)\omega$, then~$\Vert \chi \Vert_{C^0} \leq E_1$ for some constant~$E_1>0$ independent of~$\delta$.	
\end{itemize}

We will then be able to apply Proposition \ref{prop: adjusting volume forms} in order to finish the proof.

\begin{proof}[Proof of Lemma \ref{lem: new version of the extension lemma}]
We can assume, without loss of generality, that the area-form $\omega$ is standard. Indeed, if $\omega$ is not standard then we can find a symplectomorphism from the annulus $\mathbb A_2$ endowed with $\omega$ to the annulus $\mathbb A_2$ endowed with the standard area-form. The symplectomorphisms is in particular a bi-Lipschitz (for the distance $d$ induced by the metric) map, this implies that proving the lemma for the standard area-form implies the lemma in the general case.	
	
In what follows, the~$C_i$'s,~$D_i$'s and~$E_i$'s are going to be positive constants independent of~$\delta$, the letter~$C$ will be used for a control of the~$C^0$-norm, the letter~$D$ for a control of the support and the letter~$E$ for a control on an area-form. We can assume without loss of generality that the area form~$\omega$ is~$dx\wedge dy$, where~$x$ denotes the angular coordinate on~$S^1$ and~$y$ the radius along~$\mathbb{A}_3$. We start the proof as in \cite{EPP} by using a non area-preserving diffeomorphism~$f$ extending~$\phi$ as stated in Lemma \ref{lem: smooth extension}, we will discuss this lemma in Section \ref{sec: pf of extension lemma}.

\begin{lemma}\label{lem: smooth extension}
Let~$\phi$ be a smooth embedding of an open neighborhood of~$\mathbb{A}_1$ into~$\mathbb{A}_2$, isotopic to the identity, such that~$\Vert \phi \Vert_{C^0}  \leq \delta$ for some~$\delta>0$ small enough. Then there exist constants~$C_1>0$ and~$D_1>0$ that do not depend of~$\delta$, and there exists~$f\in \text{Diff}_{0,c}\left(\mathbb{A}_{2}\right)$ such that~$f$ is supported in~$\mathbb{A}_{1+D_1\delta}$, satisfies
$$f\vert_{\mathbb{A}_{1-D_1\delta}}=\phi\vert_{\mathbb{A}_{1-D_1\delta}}$$
and
$$\Vert f \Vert_{C^0}\leq C_1\delta.$$
Moreover, if~$\phi=Id$ outside a quadrilateral~$I\times [-1,1]$ and~$f(I\times [-1,1])\subset I\times [-2,2]$ for some arc~$I\subset S^1$, then~$f$ can be chosen to be the identity outside~$I \times [-2,2]$.
\end{lemma}

Denote~$\Omega:= f^*\omega$ and define~$\mathbb{A}_-:=S^1\times [-2,0]$ and~$\mathbb{A}_+:=S^1\times [0,2]$. By the condition on the area between the curves~$S^1\times y$ and~$\phi(S^1 \times y)$ the following equalities hold:
$$\int_{\mathbb{A}_-}\omega = \int_{\mathbb{A}_-}\Omega,\quad \int_{\mathbb{A}_+}\omega = \int_{\mathbb{A}_+}\Omega.$$

We are going to adjust~$f$ by constructing~$h\in \text{Diff}_{0,c}(\mathbb{A}_2)$ such that~$h\vert_{\mathbb{A}_{1-D_3\delta}}=Id$,~$h^*\Omega=\omega$ and~$\Vert h \Vert_{C^0}\leq C_3 \Vert\phi\Vert_{C^0}$, for some~$D_3, C_3>0$. To do so we just need to provide $h_+ \colon \mathbb A_+ \to \mathbb A_+$, that is the identity near $\partial \mathbb A_2$ such that $\smash[t]{h\vert_{\mathbb A_{1-D_3 \delta}}=Id}$, $h_+^*\Omega=\omega$ and $\Vert h_+ \Vert_{C^0}\geq C_3 \Vert \phi \Vert_{C^0}$, for some $D_3, C_3 >0$ on $\mathbb A_+$. Indeed, by symmetry we can then construct $h_- \colon \mathbb A_- \to \mathbb A_-$ with the same properties. We can then define $h$ by $h_+$ on $\mathbb A_+$ and $h_-$ on $\mathbb A_-$. Now define the symplectomorphism~$\Psi=fh$, $\Psi$ has compact support in $\mathbb A_2$. The symplectomorphism~$\Psi$ might not be a Hamiltonian diffeomorphism but since the $\Flux$ of $\Psi$ is controlled by $\Vert \Psi \Vert_{C^0}$, after a Lipschitz~$C^0$-adjustment (given in Lemma \ref{lem: surjectivity of O}), we can make sure that the resulting symplectomorphism $\psi$ is a Hamiltonian diffeomorphism. This would finish the proof of the extension lemma.

\begin{center}
    \textbf{Adjusting the areas of the squares}
\end{center}

The area form $\Omega$ can have very wild behavior, that is some small ball can have very big area in comparison to the area under $\omega$ and vice-versa. In this section we apply Moser's trick to small ``squares" on the annulus in order to control the ratio of those area by a constant. This will allow us to apply Proposition \ref{prop: adjusting volume forms} later.

Divide the annulus~$S^1\times [1-3(1+C_1+D_1)\delta, 1+3(1+C_1+D_1)\delta]$ in~$N$ ``squares''~$R_1,\ldots, R_N$ for some integer~$N$ such that the squares have side length~$6(1+C_1+D_1)\delta$. We denote~$\Gamma$ the 1-skeleton of the partition by squares. Then we can find~$h_1\in \text{Diff}_c(\mathbb{A}_2)$ a~$C^0$-small diffeomorphism, which has~$C^0$-norm as small as we want such that~$h_1^*\Omega$ is equal to~$\omega$ on~$\Gamma$ (for the construction of~$h_1$ we refer to the paragraph \textbf{Adjusting~$\Omega$ on~$\Gamma$} in \cite[Section 6.3]{EPP}). Denote
$$\Omega':=h_1^*\Omega,$$
and assume that~$\Vert fh_1 \Vert_{C^0}\leq (1+C_1)\delta$ by asking~$\Vert h_1 \Vert_{C^0}\leq \delta$. Note that here we have
$$\int_{\mathbb{A}_+}\Omega'=\int_{\mathbb{A}_+}\omega.$$

\medskip

We will use the same method as in the paragraph \textbf{Adjusting the areas of the squares} of \cite{EPP} but modify it in order to apply Proposition \ref{prop: adjusting volume forms} as we wish. On a square~$R_i$ we obtain the following inequality by considering the fact that~$\Vert fh_1 \Vert_{C^0}\leq (1+C_1)\delta$ so the image of~$R_i$ by~$fh_1$ is inside a square of side length~$(8+8C_1+6D_1)\delta$ and also such that the square of side length~$(4+4C_1+6D_1)\delta$ is inside of the image of~$R_i$, this gives the following estimates
$$\dfrac{4(2+2C_1+3D_1)^2 \delta^2}{ 36(1+C_1+D_1)^2 \delta^2}\leq \dfrac{\int_{R_i}\Omega'}{\int_{R_i}\omega}\leq \dfrac{4(4+4C_1+3D_1)^2\delta^2}{36(1+C_1+D_1)^2\delta^2}$$
and after simplification, 
$$\dfrac{(2+2C_1+3D_1)^2}{9(1+C_1+D_1)^2}\leq \dfrac{\int_{R_i}\Omega'}{\int_{R_i}\omega}\leq \dfrac{(4+4C_1+3D_1)^2}{9(1+C_1+D_1)^2}<\dfrac{16(1+C_1+D_1)^2}{9(1+C_1+D_1)^2}=\dfrac{16}{9}.$$

By renaming the constant on the left-hand side and right-hand side of the inequality, and setting~$s_i:=\int_{R_i}\Omega'$,~$r_i:=\int_{R_i}\omega$ we can eventually rewrite this as: 
\begin{equation}
    0<1-A\leq \dfrac{s_i}{r_i}\leq 1+A,
    \label{2}
\end{equation}
for some constant~$1>A>0$ independent of~$\delta$. Set~$t_i:=\frac{s_i}{r_i}-1$. By \eqref{2}, 
\begin{equation}
    \vert t_i \vert \leq A<1.
    \label{3}
\end{equation}

For each~$i$ choose a non-negative function~$\overline{\rho}_i$ supported in the interior of~$R_i$ such that~$\int_{R_i}\overline{\rho}_i\omega=r_i$ and 
\begin{equation}
    \Vert \overline{\rho}_i \Vert_{C^0}\leq E_2 < \dfrac{1}{A},
    \label{4}
\end{equation}
for some constant~$E_2$ independent of~$\delta$. Define a function~$\varrho$ on~$\mathbb{A}_{1+C_1+D_1}$ by
$$\varrho=1+\sum_i t_i\overline{\rho}_i.$$

We denote $D_2:=3(1+C_1+D_1)$. By \eqref{3} and \eqref{4}, we see that~$\varrho$ is positive. Moreover,~$\varrho$ is equal to~$1$ over~$\Gamma$ and the two area forms~$\varrho \omega$ and~$\Omega'$ have the same integral on each square~$R_i$. We can thus apply part (iii) of Proposition \ref{prop: Moser's trick} to the forms~$\Omega'$ and~$\varrho\omega$ on~$S^1\times [1-D_2\delta, 1+D_2\delta]$ and the skeleton~$\Gamma$: these forms coincides near the boundary of~$S^1\times [1-D_2\delta, 1+D_2\delta]$, therefore there exists a diffeomorphism~$h_2$ with compact support in~$S^1\times [1-D_2\delta, 1+D_2\delta]$ such that~$h_2^*\Omega=\varrho\omega$ and~$\Vert h_2\Vert_{C^0}\leq 2\sqrt{2D_2}\delta$.

Note that 
$$\int_{\mathbb{A}_+}\varrho\omega=\int_{\mathbb{A}_+}\Omega'=\int_{\mathbb{A}_+}\omega.$$

\medskip

In conclusion, we constructed a diffeomorphism~$fh_1h_2:=g \in \text{Diff}_c(\mathbb{A}_{1+D_2\delta})$ such that:
\begin{itemize}
    \item the~$C^0$-norm is bounded,~$d(fh_1h_2, Id)\leq (1+C_1+2\sqrt{2D_2})\delta=:C_2 \delta$,
    \item the support is controlled,~$fh_1h_2\vert_{\mathbb{A}_{1-D_2\delta}}=\phi$,
    \item the pullback of the area form is controlled by~$\varrho\omega= (fh_1h_2)^*\omega= (1+\chi)\omega$ and~$\Vert \chi \Vert_{C^0}\leq E_1:=E_2A$.
\end{itemize}

We define
$$\Omega'':=g^*\omega.$$
We can now apply Proposition \ref{prop: adjusting volume forms} with~$M=S^1$ and the two area forms~$\omega$ and~$\Omega''$ in order to find~$h$ and finish the proof.

We describe now the two changes than we need to make in order to address the second part of Lemma \ref{lem: new version of the extension lemma}.
\begin{itemize}
\item[$\bullet$] According to Lemma \ref{lem: smooth extension} we can choose $f$ such that $f$ is the identity outside $I \times [-2, 2]$.
\item[$\bullet$] Then, $f^* \Omega$ and $\omega$ agree on the complement of $I \times [-2, 2]$. This means that upon applying Moser's trick or Proposition \ref{prop: adjusting volume forms} we can just leave the surface outside of $I \times [-2, 2]$ to be preserved.
\end{itemize}
This finishes the proof of Lemma \ref{lem: new version of the extension lemma}.

\end{proof}

To adapt the extension lemma in the continuous case, we will only describe the changes in the previous proof. The overall idea stays the same in both proofs.

\begin{proof}[Proof of Lemma \ref{lem: area-preserving extension lemma for homeo}]
We extend the continuous area-preserving embedding~$\phi$ to a compactly supported homeomorphism $f$ of~$\mathbb{A}_2$ with the help of Lemma \ref{lem: continuous extension} (analogue of Lemma \ref{lem: smooth extension}).

\begin{lemma}\label{lem: continuous extension}
Let~$\phi$ be a continuous area-preserving embedding of an open neighborhood of~$\mathbb{A}_1$ into~$\mathbb{A}_2$, isotopic to the identity, such that~$\Vert\phi \Vert_{C^0}  \leq \delta$ for some~$\delta>0$ small enough. Then there exist constants~$C_1>0$ and~$D_1>0$ that do not depend on~$\delta$ and~$f\in \text{Homeo}_{0,c}\left(\mathbb{A}_{2}\right)$ such that~$f$ is supported in~$\mathbb{A}_{1+D_1\delta}$, satisfies
$$f\vert_{\mathbb{A}_{1-D_1\delta}}=\phi\vert_{\mathbb{A}_{1-D_1\delta}}$$
and 
$$\Vert f \Vert_{C^0}\leq C_1\delta.$$

Moreover, if~$f=Id$ outside a quadrilateral~$I\times [-1,1]$ and~$f(I\times [-1,1])\subset I\times [-2,2]$ for some arc~$I\subset S^1$, then~$f$ can be chosen to be the identity outside~$I \times [-2,2]$.
\end{lemma}

Denote~$\nu:=f^*\omega$ ($\nu$ is then an OU measure) then, by the condition on nullity of the area between~$S^1 \times y$ and~$\phi(S^1 \times y)$ in Lemma \ref{lem: area-preserving extension lemma for homeo}, the following equalities hold: 
$$\omega(\mathbb{A}_-) = \nu(\mathbb{A}_-),\quad \omega(\mathbb{A}_+) = \nu(\mathbb{A}_+).$$

We proceed as before and we claim that the paragraph \textbf{Modifying the area form~$\Omega$ to find a constant estimate} is actually now easier. Indeed everything transpose at one exception, since it is not mandatory to obtain a diffeomorphism we do not have to adjust the~$\Omega$ on the skeleton~$\Gamma$. We just have to apply Proposition \ref{prop: OU proposition} directly on squares~$R_1,\ldots, R_n$ alongside the neighborhood of~$\mathbb{A}_1$. The resulting~$C^0$-small homeomorphism~$h_1$ is such that~$h_1^*\nu=\varrho \omega$ is now an area-form, this means that we can apply Proposition \ref{prop: adjusting volume forms} and finish the proof of Lemma \ref{lem: area-preserving extension lemma for homeo}.
\end{proof}


\section{Proof of the extension lemmas}\label{sec: pf of extension lemma}

The proof of Lemma \ref{lem: smooth extension} can be found, after some small adjustment, in \cite[Section 6.4]{EPP}. We present instead a proof of Lemma \ref{lem: continuous extension}, the proof is an adaptation to the continuous case of Lemma 5 in \cite[Section 6.4]{EPP}. For this we will need an adaptation of the appendix of Michael Khanevsky from the same paper.

\begin{lemma}\label{lem: continuous Khanevsky extension}
Set~$L:=S\times 0$ in~$\mathbb{A}_1$. Assume that~$\phi$ is a continuous embedding of an open neighborhood of~$L$ in~$\mathbb{A}_1$, such that~$\Vert \phi \Vert_{C^0} \leq \varepsilon$ for some~$\varepsilon$ small enough.

Then there exists a homeomorphism~$\psi$ of~$\mathbb{A}_{D\varepsilon}$ such that~$\psi=\phi$ on~$L$ and~$\Vert \psi \Vert_{C^0} \leq C\varepsilon$ for some~$D,C>0$ independent on~$\varepsilon$.

Moreover, if~$\phi=Id$ outside some arc~$I\subset L$ and~$\phi(I)\subset I\times [-1,1]$, then we can also chose~$\psi$ with~$\psi$ being the identity outside~$I\times [-1,1]$.
\end{lemma}

\begin{rk}
In Lemma 6 in \cite{EPP}, the diffeomorphism $\psi$ constructed by Khanevsky is supported in $\mathbb A_2$ while the diffeomorphism (or homeomorphism) we constructed has much smaller support. Proving this does not require more effort since everything that Khanevsky does can be done with small enough support. We point out that our stronger statement is needed in the proofs of Lemma \ref{lem: new version of the extension lemma} and Lemma \ref{lem: area-preserving extension lemma for homeo}.
\end{rk}

We recall the following corollary of the Jordan-Schönflies Theorem presented in \cite{Thomassen}.

\begin{coro}\label{coro: JS}
Let~$\Gamma$ and~$\Gamma'$ be two 2-connected plane graphs (planar graphs embedded in~$\mathbb{R}^2$) and~$g$ a homeomorphism and plane-homeomorphism (i.e.~$g$ is an isomorphism such that a cycle in~$\Gamma$ is a face boundary of~$\Gamma$ if and only if the corresponding is a face boundary of~$\Gamma'$) of~$\Gamma$ onto~$\Gamma'$. Then~$g$ can be extended to a homeomorphism of the whole plane.
\end{coro}

\begin{proof}[Proof of Lemma \ref{lem: continuous Khanevsky extension}]
We start by defining some notations and preliminary results, let~$K:=\phi(L)$, and let~$K$ be parameterized by~$\gamma: \mathbb{S}^1\rightarrow \mathbb{A}_1, \, t\mapsto \phi(t,0)$. Thus~$\gamma$ can be viewed as a map from~$L$ to~$K$ and this map has a~$C^0$-norm smaller than~$\varepsilon$.

We assume without loss of generality that the length of $L$ is 1. Let $n$ be the integer $\lfloor 1/(16 \varepsilon) \rfloor$ and let~$(x_1,0)$,~$(x_2,0),\ldots, (x_{2n},0)$ be $2n$ points evenly spaced and in cyclic order on $L$. Then the distance between to consecutive points is 
\[\dfrac{1}{2n}=\dfrac{1}{2 \left \lfloor \dfrac{1}{16 \varepsilon}\right\rfloor}=\dfrac{1}{2 \left(\dfrac{1}{16 \varepsilon}-x\right)}=\dfrac{8 \varepsilon}{1-32 \varepsilon x}\leq 8 \varepsilon +O(\varepsilon^2),\]
here $0 \leq x <1$ denotes the difference between $1/(16\varepsilon)$ and its integer part.

We denote~$S_k$ the rectangle whose vertices are~$(x_k,-4\varepsilon),(x_k,4\varepsilon),(x_{k+1},4\varepsilon)$ and~$(x_{k+1},-4\varepsilon)$. We also denote~$R_k$ the rectangle whose vertices are~$(x_k-3\varepsilon,-4\varepsilon), (x_k-3\varepsilon,4\varepsilon),(x_{k+1}+3\varepsilon,4\varepsilon)$ and~$(x_{k+1}+3\varepsilon,-4\varepsilon)$. The first observation is that since~$\Vert \phi \Vert_{C^0}\leq \varepsilon$, the curve~$K$ is strictly inscribed inside the union of the squares~$S_k$.

We define~$y_k:=\sup \{y\in\mathbb{R}, (x_k,y)\in K\}$, that is~$(x_k,y_k)$ is the highest point on the side of the square~$R_k$ that is also in~$K$. Let~$r_k$ be the line segment joining~$(x_k,y_k)$ and~$(x_{k+1},y_{k+1})$. Our goal is to find~$\psi$ a~$C^0$-small homeomorphism such that~$\psi(K)$ is the union of the segments~$r_k$. We also define~$t_k:=\gamma^{-1}(x_k,y_k)$. The second observation we make is that~$\gamma\vert_{[t_k,t_{k+1}]}$ lies inside~$R_k$, indeed we must have~$[t_k,t_{k+1}]\subset [x_k-\varepsilon, x_{k+1}+\varepsilon]$, and~$\pi_x(\gamma([t_k,t_{k+1}]))\subset [x_k-2\varepsilon, x_{k+1}+2\varepsilon]$, where~$\pi_x$ denotes the projection onto~$S^1\times \{0\}$.

\begin{figure}
\centering
\includegraphics[scale=0.85]{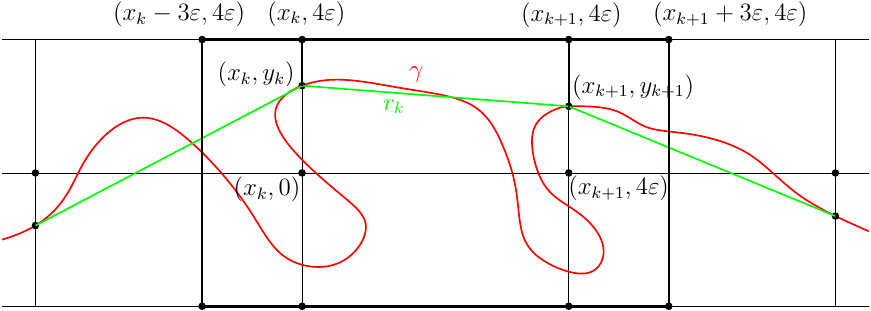}
\caption{Illustration of the proof of the extension lemma.}
\label{fig:  continuous extension}
\end{figure}

\medskip

\noindent \textbf{Step 1.} We define now, for~$k$ odd,~$\Gamma_{k}$ the 2-connected graph that coincides with the rectangle $R_k$ where we added the 4 vertices~$(x_k,4\varepsilon)$,~$(x_{k+1},4\varepsilon)$, $(x_k,y_k)$ and $(x_{k+1},y_{k+1})$ and the edges~$[(x_{k},4\varepsilon)(x_{k},y_{k})]$, $\gamma\vert_{[t_k,t_{k+1}]}$ and $[(x_{k+1},y_{k+1})(x_{k+1},4\varepsilon)]$. This is indeed a plane graph since~$\gamma\vert_{[t_k,t_{k+1}]}$ lies strictly inside~$R_k$. We also define~$\Gamma'_k$ the graph with the same edges except for~$\gamma\vert_{[t_k,t_{k+1}]}$ that is replaced by~$r_k$.
Now we define a homeomorphism~$g_k$ between~$\Gamma_k$ and~$\Gamma'_{k}$ by letting~$g$ be the identity on their common edges and any homeomorphism between~$\gamma\vert_{[t_k,t_{k+1}]}$ and~$r_k$. We can then apply the corollary \ref{coro: JS} of the Jordan-Schönflies Theorem. We obtain the homeomorphism~$\psi_k$, since the extension is made connected component by connected component we can ask for~$\psi_k$ to be the identity on the unbounded connected component of~$\Gamma_k$ and~$\Gamma'_k$. Since the support of~$\psi_k$ is inside~$R_k$ we have then~$\Vert \psi_k \Vert_{C^0}\leq 17\varepsilon$. Now~$\psi':=\Pi_{k \text{ odd}} \psi_k$ has good properties, namely,~$\Vert \psi' \Vert_{C^0}\leq 17 \varepsilon$, and~$\psi'(K)$ coincides with~$r_k$ for~$k$ odd. We denote $\check \gamma$ the curve $\psi' \circ \gamma$, the previous sentence can be rewritten $\check \gamma \vert_{[t_{k},t_{k+1}]} = r_k$ for $k$ odd.

\medskip

\noindent \textbf{Step 2.} We have to do the same thing for the other half of the curve. We define~$M_k:=((x_k+x_{k+1})/2,(y_k+y_{k+1})/2)$ and~$N_k:=((x_k+x_{k+1})/2,-4\varepsilon)$.

For this we split each segments~$r_k$ in two parts in the middle,~$p_k:=[(x_k,y_k),M_k]$ and~$q_{k+1}:=[M_k,(x_{k+1},y_{k+1})]$. For~$k$ odd we link~$M_k$ and~$N_k$ by a polygonal line segment~$s_k$ inside~$R_k$ and not crossing~$\check\gamma\vert_{[t_{k-1},t_k]}$,~$\check\gamma\vert_{[t_k,t_{k+1}]}$ nor~$\check\gamma\vert_{[t_{k+1},t_{k+2}]}$ (this can be done by means of Lemma 2.1 of \cite{Thomassen} for example). We define now, for~$k$ even,~$R'_k$ the polygon whose sides are the same as of~$S_k$ but we replace the segment~$[(x_{k},y_{k})(x_{k},-4\varepsilon)]$ by~$q_k\cup s_{k-1} \cup [N_{k-1}(x_{k},-4\varepsilon)]$ and the line segment~$[(x_{k+1},y_{k+1})(x_{k+1},-4\varepsilon)]$ by~$p_{k+1}\cup s_{k+1} \cup [N_{k+1}(x_{k+1},-4\varepsilon)]$. Since the curve $\check{\gamma}$ avoids the quadrilateral $Q_{k+1}$ formed by the vertices $(x_{k+1}, y_{k+1})$, $(x_{k+2}, y_{k+2})$, $(x_{k+2}, 4\varepsilon)$ and $(x_{k+1}, 4\varepsilon)$ there is some $\eta >0$ small enough such that, for all $k$ even, the segments $r'_k=[(x_{k}, 4\varepsilon-\eta)(x_{k+1}, 4\varepsilon-\eta)]$ avoid $\check{\gamma}$.

We can now apply Corollary \ref{coro: JS} to the 2-connected plane graphs
\[\Gamma_k:=R'_k \cup \check\gamma\vert_{[t_{k},t_{k+1}]}\cup Q_{k-1}\cup Q_{k+1}\]
and
\[\Gamma'_k:=R'_k \cup r'_k \cup Q_{k-1} \cup Q_{k+1}\]
and the homeomorphism~$g$ being the identity on the exterior on the graphs, $g((x_{k-1},y_{k-1}))=(x_{k-1}, 4\varepsilon-\eta)$ and $g((x_{k+1},y_{k+1}))=(x_{k+1}, 4\varepsilon-\eta)$. We obtain~$\psi_k$ a homeomorphism such that~$\psi_k(\check\gamma\vert_{[t_{k},t_{k+1}]})$ is the segment $r'_k$, $\psi_k(r_{k+1})$ is the broken segment composed of $[(x_{k+1},4\varepsilon-\eta)(x_{k+1}, y_{k+1})]$ and $r_{k+1}$, and~$\psi_k$ is the identity outside~$R'_k$. It becomes now easy to modify $\psi_k$ to obtain $\psi'_k$ such that $\psi'_k(\check{\gamma}\vert_{[t_{k-1},t_{k+2}]})$ is the broken segment $r_{k-1} \cup r_k \cup r_{k+1}$.

We denote~$\psi''=\Pi_{k \text{ even}}\psi'_k$, since the diameter of~$R'_k$ is smaller than~$C\varepsilon$ for~$C$ a positive real number we have~$\Vert \psi'_k \Vert_{C^0}\leq C \varepsilon$ and~$\Vert \psi'' \Vert_{C^0}\leq C\varepsilon$. 
and~$\psi'' (\psi'(K))$ is now the union of the segments~$r_k$.

\medskip

\noindent \textbf{Step 3.} Once we have the union of the segments~$r_k$ it is easier to finish, by taking for example the flow of an appropriate vector field (the cut-off of a constant vector field on the fibers~$(x,\cdot)$ for example). We denote~$\psi'''$ the last homeomorphism we obtain.

\medskip

\noindent \textbf{Step 4.} There is one last thing to do is a small perturbation of~$\psi''' \circ \psi'' \circ \psi'$ such that it coincides on~$L$ with the homeomorphism~$\phi$. This finishes the proof.

\end{proof}

This is the tool we needed to prove Lemma \ref{lem: continuous extension}. Notice here that in contrary of the proof of Lemma 6.6 in \cite{EPP} the proof is very short, indeed we do not need to care more about the extension since we care about obtaining a diffeomorphism and not only a homeomorphism.

\begin{proof}[Proof of Lemma \ref{lem: continuous extension}]
We apply Lemma \ref{lem: continuous Khanevsky extension} to the curves~$S^1\times \{\pm1 \}$ in~$\mathbb{A}_2$ and their images under~$\phi$. We can find~$\psi \in \text{Homeo}_{0,c}(\mathbb{A}_{2})$ supported in~$\mathbb{A}_{1-D_1\delta,1+D_1\delta}$ and agreeing with~$\phi^{-1}$ on~$\phi(S \times \{\pm 1\})$. Now when we extend~$\psi'=\psi \phi$ by the identity outside of~$\mathbb{A}_1$ we get the required result. 
\end{proof}


\section{Annex: Annex: Discussion on the constant in Corollary \ref{coro: $C^0$-small isotopy}}

In this section we discuss the question of finding the optimal constant in Corollary \ref{coro: $C^0$-small isotopy}.

\begin{prop}\label{prop: counterexample to constant 1}
	For all $C>1$, there exists a metric $g_C$ on the standard symplectic torus $(\mathbb T^2, \omega)$ such that there exists an area-preserving diffeomorphism $\varphi$ in the connected component of the identity but no isotopy $\varphi_t$ such that $\varphi_0=Id$, $\varphi_1=\varphi$ and for all $t$, $\Vert \varphi_t \Vert_{C^0}\leq C \Vert \varphi \Vert_{C^0}$.
\end{prop}

The example we give sadly only works at large scale and we use global topological properties of the manifold. The author asks the following question.

\begin{ques}
	For all real number $C \geq 1$ and a surface $\Sigma$, does there exists a metric $g_C$ such that for all $\varepsilon >0$, there exists an area-preserving diffeomorphism $\varphi$ in the connected component of the identity that satisfies that $\Vert \varphi \Vert_{C^0} \leq \varepsilon$ but no isotopy (not necessarily symplectic) $\varphi_t$ such that $\varphi_0=Id$, $\varphi_1=\varphi$ and for all time $t$, $\Vert \varphi_t \Vert_{C^0} \leq C \Vert \varphi \Vert_{C^0}$?
\end{ques}

Let us describe the example.

\begin{proof}[Proof of Proposition \ref{prop: counterexample to constant 1}]
	Pick two real numbers $0 < a < b$. We denote $\mathcal C_1$ and $\mathcal C_3$ two copies of the flat cylinder of circumference $2a$ and length $2a$, that is $\mathcal C_1$ is isometric to the quotient of $[-a; a]\times \mathbb R$ by the translation $x \mapsto x+2a$ on the second factor. We also denote $\mathcal C_2$ and $\mathcal C_4$ two copies of the flat cylinder of circumference $2b$ and length $2b$. We glue those four cylinders together to form a torus $\mathbb T_{a,b}$ in the following order $\mathcal C_1$, $\mathcal C_2$, $\mathcal C_3$ and finally $\mathcal C_4$ by adding small cylinders that interpolate the flat metric of the cylinders $\mathcal C_i$. We define $\varphi$ as two area-preserving Dehn twist on the cylinders $\mathcal C_1$ and $\mathcal C_3$ in opposite directions, this way the map is isotopic to the identity. Moreover, since $\varphi$ is compactly supported in $\mathcal C_1 \cup \mathcal C_3$, we get that $\Vert \varphi_t \Vert_{C^0} \leq a$. We pick an isotopy $\varphi_t$ and we assume by contradiction that for all $t$, $\Vert \varphi_t \Vert_{C^0} < b$.
	
	\begin{figure}[!h]
		\centering
		\includegraphics[scale=0.55]{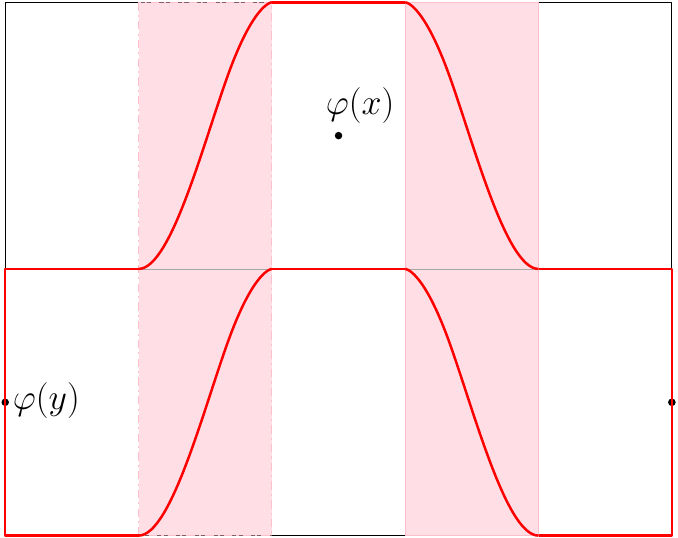}
		\caption{We have represented two fundamental domains. In red we drew the boundary of the image of a fundamental domain under $\varphi$, the cylinders $\mathcal C_1$ and $\mathcal C_3$ are colored in pink}
		\label{Fig: counterexample}
	\end{figure}
	
	Let $x$ and $y$ be points on the equator of $\mathcal C_2$ and $\mathcal C_4$ respectively. The loop $\varphi_t(x)$ stays in the open ball of radius $b$, thus in $\mathcal C_2$ and this means that the loop is contractible. We can apply the same
	reasoning to the loop $\varphi_t(y)$ in the cylinder $\mathcal C_4$. However, because of the Dehn twists it is impossible that both loops are contractible (see Remark \ref{rk: more explanation} for more explanation). This is the contradiction we wanted. This implies that for some $t$, $\Vert \varphi_t \Vert_{C^0} \geq b$.
	It suffices now to pick $a$ and $b$ such that $b/a > C$ to conclude the proof of the proposition.
\end{proof}

\begin{rk}\label{rk: more explanation}
A nice way to see that the last argument is true is to look at the image of a fundamental
domain $\mathcal D$ of $\mathbb T_{a,b}$ under any lift of the map $\varphi$ to the fundamental domain of $\mathbb T_{a,b}$ and to see that the image of the points $x$ and $y$ do not lie both in $\mathcal D$. See Figure \ref{Fig: counterexample} for an illustration.

\end{rk}

\bibliographystyle{alpha}
\bibliography{document}

\begin{thebibliography}{CCGF}

\bibitem[Ali24]{Ali} H.~Alizadeh, \emph{Hamiltonian fragmentation in dimension four with application to spectral estimators}, Journal
of Topology and Analysis (2024), 1--21.

\bibitem[Ban78]{Ban78} A.~Banyaga,  {\it Sur la structure du groupe des
diff{\'e}omorphismes qui pr{\'e}servent une forme symplectique},
Comm. Math. Helv.~{\bf 53} (1978), 174--227.

\bibitem[Ban97]{Ban97}A.~Banyaga, \emph{The structure of classical diffeomorphism groups}, Mathematics
and its applications \textbf{400}, Kluwer Academic Publishers Group,
Dordrecht, 1997.

\bibitem[CHS24]{CriHumSey} D.~Cristofaro-Gardiner, V.~Humilière, S.~Seyfaddini, \emph{Proof of the symplicity conjecture}, Annals of Mathematics \textbf{199} (2024), 181--257.

\bibitem[EPP12]{EPP} M.~Entov, L.~Polterovich, P.~Py, \emph{On continuity of quasimorphisms for symplectic maps}, Progress in Math.~ \textbf{296} (2012), 169--197.

\bibitem[Fa80]{Fathi} A.~Fathi, \emph{Structure of the group of homeomorphisms preserving a
good measure on a compact manifold}, Annales scientifiques de l’\'E.N.S. (4), \textbf{13}, No. 1 (1980), 45--93.

\bibitem[Lef]{Lefeuvre} T.~Lefeuvre, \emph{Uniform Approximation of Volume-preserving Homeomorphisms by Volume-preserving Diffeomorphisms}. unpublished.

\bibitem[LeR10]{LeRoux} F.~Le Roux, \emph{Simplicity of~$\text{Homeo}(D^2, \partial D^2, Area)$ and fragmentation of symplectic diffeomorphisms}, J. Symplectic Geom. \textbf{8}, No.1, 2010, 73--93.

\bibitem[MS98]{McD-Sal} D.~McDuff, D.~Salamon, {\it Introduction to
symplectic topology}, 2nd edition, Oxford University Press,
Oxford, 1998.

\bibitem[Mos65]{Moser} J.~Moser, \emph{On the volume elements on a
manifold}, Trans. Amer. Math. Soc.~\textbf{120} (1965), 288--294.

\bibitem[Oh10]{Oh} Y.-G.~Oh, \emph{The group of Hamiltonian homeomorphisms and continuous Hamiltonian flows}, , Contemp. Math., \textbf{512}, Amer. Math. Soc., Providence, RI, (2010), 149--177.

\bibitem[OU41]{OU} J.~C.~Oxtoby, S.~M.~Ulam, \emph{Measure-Preserving Homeomorphisms and Metrical Transitivity}, Annals of Mathematics, Second Series, \textbf{42}, (4) (1941), 874--920.

\bibitem[Sch57]{Schwartzman} S,~Schwartzman, \emph{Asymptotic Cycles}, Annals of Mathematics , Second Series, \textbf{66}, No. 2 (1957),
270--284.

\bibitem[Sey13]{Seyfaddini} S.~Seyfaddini, \emph{$C^0$-limits of hamiltonian paths and the Oh-Schwarz spectral invariants}, Int. Math. Res. Not. ~\textbf{21} (2013), 4920--4960.

\bibitem[Tho92]{Thomassen} C.~Thomassen \emph{The Jordan-Schönflies Theorem and the Classification of Surfaces}, The American Mathematical Monthly, \textbf{99}, No. 2 (1992), 116--131.

\bibitem[Thu]{Thurston} W.~Thurston, \emph{On the structure of the volume preserving diffeomorphisms}, unpublished.

\end{thebibliography}

\end{document}